\documentclass[12pt]{article}

\usepackage{amssymb}
\usepackage{amsfonts}
\usepackage{dsfont}
\usepackage{amsmath}
\usepackage{amsthm}

\newtheorem{theorem}{Theorem}[section]
\newtheorem{lemma}[theorem]{Lemma}
\newtheorem{remark}[theorem]{Remark}

\newtheorem{proposition}[theorem]{Proposition}

\newtheorem{corollary}[theorem]{Corollary}

\newcommand{\bZ}{\mathbb{Z}}
\newcommand{\bN}{\mathbb{N}}
\newcommand{\bR}{\mathbb{R}}
\newcommand{\bC}{\mathbb{C}}
\newcommand{\bE}{\mathbb{E}}

\newcommand{\cF}{\mathcal{F}}
\newcommand{\Cov}{{\rm Cov\,}}

\newcommand{\Prob}{\mathds{P}}

\title{Strictly stationary solutions of multivariate ARMA equations with i.i.d. noise}
\author{Peter J. Brockwell\thanks{Colorado State University, Fort Collins, Colorado and Columbia University, New York. \texttt{pjb2141@columbia.edu}. } \and
Alexander Lindner\thanks{Institut f\"ur Mathematische Stochastik, TU
Braunschweig,  Pockelsstra{\ss}e 14, D-38106 Braunschweig, Germany
\texttt{a.lindner@tu-bs.de}} \and{Bernd Vollenbr\"oker}\thanks{Institut f\"ur Mathematische Stochastik, TU
Braunschweig,  Pockelsstra{\ss}e 14, D-38106 Braunschweig, Germany
\texttt{b.vollenbroeker@tu-bs.de}}}

\begin{document}
\maketitle

\begin{abstract}We obtain necessary and sufficient conditions for
the existence of strictly stationary solutions of multivariate ARMA
equations with independent and identically distributed noise.
For general ARMA$(p,q)$ equations
these conditions are expressed in terms of the characteristic
polynomials of the defining equations and moments of the driving noise
sequence, while for $p=1$ an additional characterization is obtained
in terms of the Jordan canonical decomposition of the autoregressive
matrix, the moving average coefficient matrices and the noise sequence.
No a priori assumptions are made on either the driving noise sequence or the
coefficient matrices.
\end{abstract}

\section{Introduction}
Let $m,d\in \bN=\{1,2\ldots, \}$, $p,q\in \bN_0 = \bN \cup \{0\}$,
$(Z_t)_{t\in \bZ}$ be a $d$-variate noise sequence of random vectors
defined on some probability space $(\Omega,\cF,\mathbb{P})$ and
$\Psi_1,\ldots, \Psi_p \in \mathbb{C}^{m\times m}$ and
$\Theta_0,\ldots, \Theta_q \in \mathbb{C}^{m\times d}$ be
deterministic complex-valued matrices. Then any $m$-variate
stochastic process $(Y_t)_{t\in \bZ}$ defined on the same
probability space $(\Omega,\cF,\mathbb{P})$ which satisfies almost
surely
 \begin{equation}\label{eqpq} Y_t-\Psi_1 Y_{t-1}-\ldots-\Psi_p Y_{t-p}=\Theta_0Z_t+\ldots+\Theta_qZ_{t-q},\quad
 t\in\mathds{Z},\end{equation}
is called a solution of the ARMA$(p,q)$ equation \eqref{eqpq}
(autoregressive moving average equation of autoregressive order $p$
and moving average order $q$). Such a solution is often called a VARMA (vector ARMA) process to distinguish it from the scalar case, but we shall simply use the term ARMA throughout.
Denoting the identity matrix in $\bC^{m\times m}$ by ${\rm Id}_m$, the {\it characteristic
polynomials} $P(z)$ and $Q(z)$ of the ARMA$(p,q)$ equation
\eqref{eqpq} are defined as
\begin{equation} \label{eq-char}
P(z) := \mbox{\rm Id}_m - \sum_{k=1}^p  \Psi_k z^k \quad \mbox{and}
\quad Q(z) := \sum_{k=0}^q \Theta_k z^k\quad \mbox{for}\quad z\in
\bC.
\end{equation}
With the aid of the backwards shift operator $B$, equation
\eqref{eqpq} can be written  more compactly in the form
\begin{equation*} \label{ARMA}
P(B) Y_t = Q(B) Z_t, \quad t\in \bZ.
\end{equation*}

There is evidence to show that, although VARMA($p,q$) models with $q>0$ are more difficult to estimate
than VARMA$(p,0)$ (vector autoregressive) models, significant improvement in forecasting performance can be achieved by allowing the moving average order $q$ to be greater than zero.  See, for example, Athanosopoulos and Vahid~\cite{AV}, where such improvement is demonstrated for a variety of macroeconomic time series.

Much attention has been paid to {\it weak ARMA processes}, i.e.
weakly stationary solutions to \eqref{eqpq} if $(Z_t)_{t\in \bZ}$ is
a weak white noise sequence. Recall that a $\bC^r$-valued process
$(X_t)_{t\in \bZ}$ is {\it weakly stationary} if each $X_t$ has
finite second moment, and if $\mathbb{E} X_t$ and $\Cov (X_t,
X_{t+h})$ do not depend on $t\in \bZ$ for each $h\in \bZ$. If
additionally every component of $X_t$ is uncorrelated with every
component of $X_{t'}$ for $t\neq t'$, then $(X_t)_{t\in \bZ}$ is
called {\it weak white noise}. In the case when $m=d=1$ and $Z_t$ is
weak white noise having non-zero variance, it can easily be shown using
spectral analysis, see e.g. Brockwell and Davis~\cite{BD},
Problem~4.28, that a weak ARMA process exists if and only if the
rational function $z\mapsto Q(z) / P(z)$ has only removable
singularities on the unit circle in $\bC$. For higher dimensions, it
is well known that a sufficient condition for weak ARMA processes to
exist is that the polynomial $z \mapsto \det P(z)$ has no zeroes on
the unit circle (this follows as in Theorem~11.3.1 of Brockwell and
Davis~\cite{BD}, by developing $P^{-1}(z) = (\det P(z))^{-1}
\mbox{Adj} (P(z))$, where $\mbox{Adj}( P(z))$ denotes the adjugate
matrix of $P(z)$, into a Laurent series which is convergent in a
neighborhood of the unit circle). However,  to the best of our
knowledge necessary and sufficient conditions have not been given in
the literature so far. We shall obtain such a condition in terms of
the matrix rational function $z\mapsto P^{-1}(z) Q(z)$ in
Theorem~\ref{thm-4}, the proof being an easy extension of the
corresponding one-dimensional result.

Weak ARMA processes, by definition, are restricted to have finite second moments.
However financial time series often exhibit apparent heavy-tailed behaviour with
asymmetric marginal distributions, so that second-order properties are inadequate to
account for the data. To deal with such phenomena we focus in this paper on {\it strict ARMA
processes}, by which we mean strictly stationary solutions of
\eqref{eqpq} when $(Z_t)_{t\in \bZ}$ is supposed to be an
independent and identically distributed (i.i.d.) sequence of random
vectors, not necessarily with finite variance. A sequence
$(X_t)_{t\in \bZ}$ is {\it strictly stationary} if all its finite
dimensional distributions are shift invariant. Much less is known
about strict ARMA processes, and it was shown only recently for
$m=d=1$ in Brockwell and Lindner \cite{BL2} that for i.i.d.
non-deterministic noise $(Z_t)_{t\in \bZ}$, a strictly stationary
solution to \eqref{eqpq} exists if and only if $Q(z)/P(z)$ has only
removable singularities on the unit circle and $Z_0$ has finite log
moment, or if $Q(z)/P(z)$ is a polynomial. For higher dimensions,
while it is known that finite log moment of $Z_0$ together with
$\det P(z) \neq 0$ for $|z|=1$ is {\it sufficient} for a strictly
stationary solution to exist, by the same arguments used for
weakly stationary solutions, necessary and sufficient conditions
have not been available so far, and we shall obtain a complete
solution to this question in Theorem~\ref{thm-5}, thus generalizing
the results of \cite{BL2} to higher dimensions. A related question
was considered by Bougerol and Picard~\cite{BP} who, using their powerful
results on random recurrence
equations, showed in Theorem~4.1 of \cite{BP} that if
$\mathbb{E} \log^+\|Z_0\| < \infty$ and the characteristic
polynomials are left-coprime, meaning that the only common left-divisors
of $P(z)$ and $Q(z)$ are unimodular (see Section~\ref{S7}
for the precise definitions), then a non-anticipative strictly
stationary solution to \eqref{eqpq} exists if and only if $\det P(z)
\neq 0$ for $|z|\leq 1$. Observe that for the characterization of
the existence of strict (not necessarily non-anticipative) ARMA
processes obtained in the present paper, we shall not make any a
priori assumptions on log moments of the noise sequence or on
left-coprimeness of the characteristic polynomials, but rather obtain
related conditions as parts of our characterization. As an
application of our main results, we shall then obtain a slight
extension of Theorem~4.1 of Bougerol and Picard~\cite{BP} in
Theorem~\ref{cor-BP}, by characterizing all non-anticipative
strictly stationary solutions to \eqref{eqpq} without any moment
assumptions, however still assuming left-coprimeness of the
characteristic polynomials.

The paper is organized as follows. In Section~\ref{S2} we state the
main results of the paper. Theorem~\ref{thm-main} gives necessary
and sufficient conditions for the
multivariate ARMA$(1,q)$ model
\begin{equation} \label{eq1q}
Y_t -
\Psi_1 Y_{t-1} = \sum_{k=0}^q \Theta_k Z_{t-k}, \quad t\in \bZ,
\end{equation}
where $(Z_t)_{t\in \bZ}$ is an i.i.d. sequence, to have a strictly
stationary solution. Elementary considerations will show that the
question of strictly stationary solutions may be reduced to the
corresponding question when $\Psi_1$ is assumed to be in Jordan
block form, and Theorem~\ref{thm-main} gives a characterization of
the existence of strictly stationary ARMA$(1,q)$ processes in terms
of the Jordan canonical decomposition of $\Psi_1$ and properties of
$Z_0$ and the coefficients $\Theta_k$.
An explicit solution of \eqref{eq1q},
assuming its existence, is also derived and the question of uniqueness
of this solution is addressed.

Strict ARMA$(p,q)$ processes are addressed in
Theorem~\ref{thm-5}. Since every $m$-variate ARMA$(p,q)$ process can
be expressed in terms of a corresponding $mp$-variate ARMA$(1,q)$ process,
questions of existence and uniqueness can, in principle, be resolved by
Theorem~\ref{thm-main}. However,
since the Jordan canonical form of the corresponding $mp\times mp$-matrix
$\underline{\Psi}_1$ in the corresponding higher-dimensional
ARMA$(1,q)$ representation is in general difficult to handle,
another more compact characterization is derived in
Theorem~\ref{thm-5}. This characterization is  given in terms of
properties of the matrix rational function $P^{-1}(z) Q(z)$ and
finite log moments of certain linear combinations of the components
of $Z_0$, extending the corresponding condition obtained in
\cite{BL2} for $m=d=1$ in a natural way. Although in the statement
of Theorem~\ref{thm-5} no transformation to Jordan canonical forms
is needed, its proof makes fundamental use of
Theorem~\ref{thm-main}.

Theorem~\ref{thm-4} deals with the corresponding question for weak
ARMA$(p,q)$ processes. The proofs of Theorems~\ref{thm-main},
\ref{thm-4} and \ref{thm-5} are given in Sections~\ref{S5}, \ref{S4}
and \ref{S6}, respectively.  The proof of Theorem~\ref{thm-5} makes
crucial use of Theorems~\ref{thm-main} and~\ref{thm-4}.

The main results  are further discussed in Section~\ref{S7} and, as
an application, the aforementioned characterization of
non-anticipative strictly stationary solutions is obtained in
Theorem~\ref{cor-BP}, generalizing slightly the result of Bougerol
and Picard~\cite{BP}.

Throughout the paper, vectors will be understood as column vectors
and $e_i$ will denote the $i^{th}$ unit vector in $\bC^m$. The zero
matrix in $\bC^{m\times r}$ is denoted by $0_{m,r}$ or simply $0$,
the zero vector in $\bC^r$ by $0_r$ or simply $0$. The transpose of
a matrix $A$ is denoted by $A^T$, and its complex conjugate
transpose matrix by $A^* = \overline{A}^T$. By
 $\| \cdot \|$ we denote an unspecific, but fixed vector
 norm on $\bC^s$ for $s\in \bN$, as well as
the corresponding matrix norm $\|A\| = \sup_{x\in \bC^s, \|x\|=1}
\|Ax\|$. We write $\log^+ (x) := \log \max \{1,x\}$ for $x\in \bR$,
and denote by $\mathbb{P}-\lim$ limits in probability.

\section{Main results} \label{S2}
\setcounter{equation}{0}

Theorems~\ref{thm-main} and~\ref{thm-5} give necessary and
sufficient conditions for the ARMA$(1,q)$ equation \eqref{eq1q} and
the ARMA$(p,q)$ equation \eqref{eqpq}, respectively,  to have a strictly
stationary solution. In Theorem~\ref{thm-main}, these conditions are
expressed in terms of the i.i.d. noise sequence $(Z_t)_{t\in
\mathbb{Z}}$, the coefficient matrices $\Theta_0, \ldots, \Theta_q$
and the Jordan canonical decomposition of $\Psi_1$, while in
Theorem~\ref{thm-5} they are given in terms of the noise sequence
and the characteristic polynomials $P(z)$ and $Q(z)$ as defined in
\eqref{eq-char}.

As background for Theorem~\ref{thm-main}, suppose that $\Psi_1\in
\bC^{m\times m}$ and choose a (necessarily non-singular) matrix
$S\in \bC^{m\times m}$ such that
 $S^{-1}\Psi_1 S$ is in Jordan canonical form.
   Suppose also that $S^{-1}\Psi_1 S$ has $H\in \bN$ Jordan blocks,
   $\Phi_1,\ldots,\Phi_H$, the $h^{th}$ block beginning in row $r_h$, where
 $r_1:=1<r_2<\cdots<r_{H}< m+1=:r_{H+1}.$ A Jordan block with
 associated eigenvalue $\lambda$ will always be understood to be of the form
 \begin{equation}\begin{pmatrix} \lambda &  & & 0 \\
 1 &  \lambda  & & \\
 &  \ddots & \ddots & \\
 0 &  &  1 & \lambda
 \end{pmatrix} \label{eq-Jordanblock2} \end{equation}
i.e. the entries 1 are below the main diagonal.

Observe that (\ref{eq1q}) has a strictly stationary solution
$(Y_t)_{t\in \bZ}$
 if and only if the corresponding equation for
 $X_t:=S^{-1}Y_t$ namely
 \begin{equation}\label{diag1q}X_t-S^{-1}\Psi_1 SX_{t-1}=\sum_{j=0}^q S^{-1}\Theta_jZ_{t-j},\quad t\in \bZ,\end{equation}
 has a strictly stationary solution. This will be the case only if the equation for the $h^{th}$ block,
 \begin{equation} X_t^{(h)}:=I_h X_t, \quad t\in \bZ, \label{eq-Xt}
 \end{equation}
 where $I_h$ is the $(r_{h+1}-r_h)\times m$ matrix with $(i,j)$ components,
 \begin{equation} \label{def-Is}
 I_h(i,j)=\begin{cases}1,&{\rm if}~j=i+r_h-1,\cr
                         0,&{\rm otherwise},\cr
            \end{cases}\end{equation}
 has a strictly stationary solution for each $h=1,\ldots,H.$   But these equations are simply
 \begin{equation}\label{block1q}X_t^{(h)}-\Phi_h X_{t-1}^{(h)}=\sum_{j=0}^q I_h S^{-1}\Theta_jZ_{n-j}, \quad t\in \bZ, \quad
 h=1,\ldots,H,\end{equation}
 where $\Phi_h$ is the $h^{th}$ Jordan block of $S^{-1}\Psi_1 S$.

 Conversely if (\ref{block1q}) has a strictly stationary solution ${X'}^{(h)}$ for each
 $h\in\{1,\ldots,H\}$,
 then we shall see from the proof of Theorem~\ref{thm-main} that there exist (possibly different if $|\lambda_h|=1$) strictly stationary solutions
 $X^{(h)}$ of \eqref{block1q} for each $h\in \{1,\ldots, H\}$, such that
 \begin{equation}Y_t:=S(X_t^{(1)T}, \ldots, X_t^{(H)T})^T, \quad t\in \bZ, \label{eq-solution}\end{equation}
 is a strictly stationary solution of (\ref{eq1q}).

 Existence and uniqueness of a strictly stationary solution of (\ref{eq1q}) is therefore equivalent to the
  existence and uniqueness of a strictly stationary solution of the equations (\ref{block1q}) for
   each $h\in\{1,\ldots,H\}$. The necessary and sufficient condition for each one will depend on the value of the eigenvalue $\lambda_h$
   associated with $\Phi_h$ and in particular on whether (a) $|\lambda_h|\in(0,1)$,
 (b) $|\lambda_h|>1$, (c) $|\lambda_h|=1$ and $\lambda_h\ne 1$, (d) $\lambda_h=1$ and (e) $\lambda_h=0$.  These cases will be addressed separately in the proof of
 Theorem~\ref{thm-main}, which is given in Section~\ref{S5}. The
 aforementioned characterization in terms of the Jordan decomposition of $\Psi_1$ now reads as follows.

\begin{theorem} {\rm [Strict ARMA$(1,q)$ processes]}
\label{thm-main} \\ Let $m,d\in \bN$, $q\in \bN_0$, and let
$(Z_t)_{t\in \bZ}$ be an i.i.d. sequence of $\bC^d$-valued random
vectors. Let $\Psi_1 \in \bC^{m\times m}$ and $\Theta_0, \ldots,
\Theta_q \in \bC^{m\times d}$ be complex-valued matrices. Let $S\in
\bC^{m\times m}$ be an invertible matrix such that $S^{-1} \Psi_1 S$
is in Jordan block form as above, with $H$ Jordan blocks $\Phi_h$,
$h\in \{1,\ldots, H\}$, and associated eigenvalues $\lambda_h$,
$h\in \{1,\ldots, H\}$. Let $r_1, \ldots, r_{H+1}$ be given as above
and $I_h$ as defined by \eqref{def-Is}. Then the ARMA$(1,q)$
equation \eqref{eq1q} has a strictly stationary solution $Y$ if and
only if the following statements (i) -- (iii) hold:
\begin{enumerate}
\addtolength{\itemsep}{-1.5ex}
\item[(i)] For every $h\in \{1,\ldots, H\}$ such that $|\lambda_h|
\neq 0,1$,
\begin{equation}
\mathbb{E} \log^+ \left\| \left( \sum_{k=0}^q \Phi_{h}^{q-k}
I_{h}S^{-1}\Theta_k\right)Z_0 \right\|  <  \infty.\label{bed1a}
\end{equation}
\item[(ii)] For every $h\in \{1,\ldots, H\}$ such that
$|\lambda_h|=1$, but $\lambda_h \neq 1$, there exists a constant
$\alpha_h \in \bC^{r_{h+1}-r_h}$ such that
\begin{equation} \label{bed2a}
\left(\sum_{k=0}^q \Phi_{h}^{q-k} I_{h} S^{-1} \Theta_k\right)Z_0 =
 {\alpha}_h \; \; \mbox{\rm a.s.}
\end{equation}
\item[(iii)] For every $h\in \{1,\ldots, H\}$ such that
$\lambda_h=1$, there exists a constant $\alpha_h = (\alpha_{h,1},
\ldots,$ $\alpha_{h,r_{h+1}-r_h})^T  \in \bC^{r_{h+1}-r_h}$ such
that $\alpha_{h,1} = 0$ and \eqref{bed2a} holds.
\end{enumerate}
 If these conditions are
satisfied, then a strictly stationary solution to \eqref{eq1q} is
given by \eqref{eq-solution} with
\begin{equation} \label{eq-solution1}
X_t^{(h)} := \begin{cases} \sum_{j=0}^{\infty} \Phi_{h}^{j-q} \left(
\sum_{k=0}^{j\wedge q} \Phi_{h}^{q-k} I_{h} S^{-1}
\Theta_k\right) Z_{t-j} , & |\lambda_h| \in (0,1), \\
 -\sum_{j=1-q}^{\infty} \Phi_{h}^{-j-q} \left( \sum_{k=(1-j)\vee
0}^{q} \Phi_{h}^{q-k} I_{h} S^{-1}
\Theta_k\right) Z_{t+j} , & |\lambda_h| > 1\\
\sum_{j=0}^{m+q-1} \left( \sum_{k=0}^{j\wedge q} \Phi_h^{j-k} I_h
S^{-1}\Theta_k \right)Z_{t-j}, & \lambda_h = 0,\\
f_h + \sum_{j=0}^{q-1} \left
(\sum_{k=0}^j\Phi_{h}^{j-k}I_{h}S^{-1}\Theta_k\right) Z_{t-j}, &
|\lambda_h| =1,\end{cases} \end{equation} where $f_h \in
\bC^{r_{h+1}-r_h}$ is a solution to
\begin{equation} \label{eq-solution6}
(\mbox{\rm Id}_{h} - \Phi_h)f_h = \alpha_h,
\end{equation}
which exists for $\lambda_h=1$ by (iii) and, for $|\lambda|=1,\lambda\ne 1$, by the invertibility of
$(\mbox{\rm Id}_{h} - \Phi_h)$.
 The series in \eqref{eq-solution1} converge a.s. absolutely.\\
If the necessary and sufficient conditions stated above are
satisfied, then, provided the underlying probability space is rich
enough to support a random variable which is uniformly distributed
on $[0,1)$ and independent of $(Z_t)_{t\in \bZ}$, the solution given
by \eqref{eq-solution} and \eqref{eq-solution1} is the unique
strictly stationary solution of \eqref{eq1q} if and only if
$|\lambda_h| \neq 1$ for all $h\in \{1,\ldots,H\}$.
 \end{theorem}

Special cases of Theorem~\ref{thm-main} will be treated in Corollaries
\ref{cor-1}, \ref{cor-2} and Remark \ref{rem-2}.

It is well known that every ARMA$(p,q)$ process can be embedded into
a higher dimensional ARMA$(1,q)$ process as specified in
Proposition~\ref{thm-2} of Section~\ref{S6}. Hence, in principle,
the questions of existence and uniqueness of strictly stationary ARMA$(p,q)$ processes can be
reduced to Theorem~\ref{thm-main}. However, it is generally
difficult to obtain the Jordan canonical decomposition of the
$(mp\times mp)$-dimensional matrix $\underline{\Phi}$ defined in
Proposition~\ref{thm-2}, which is needed to apply
Theorem~\ref{thm-main}. Hence, a more natural approach is to express
the conditions in terms of the characteristic polynomials $P(z)$ and
$Q(z)$ of the ARMA$(p,q)$ equation \eqref{eqpq}.  Observe that $z\mapsto \det P(z)$ is a polynomial in $z\in \bC$, not
identical to the zero polynomial. Hence $P(z)$ is invertible except
for a finite number of $z$. Also, denoting the adjugate matrix of
$P(z)$ by $\mbox{Adj} (P(z))$, it follows from Cram\'er's inversion
rule that the inverse $P^{-1}(z)$ of $P(z)$ may be written as
 $$P^{-1}(z) = (\det P(z))^{-1} \mbox{Adj} (P(z))$$ which is a
$\bC^{m\times m}$-valued rational function, i.e. all its entries are
rational functions. For a general matrix-valued rational function
$z\mapsto M(z)$ of the form $M(z) = P^{-1}(z) \widetilde{Q}(z)$ with
some matrix polynomial $\widetilde{Q}(z)$, the {\it singularities}
of $M(z)$ are the zeroes of $\det P(z)$, and such a singularity,
$z_0$ say, is {\it removable} if all entries of $M(z)$ have
removable singularities at $z_0$. Further observe that if $M(z)$ has
only removable singularities on the unit circle in $\bC$, then
$M(z)$ can be expanded in a Laurent series $M(z) =
\sum_{j=-\infty}^\infty M_j z^j$, convergent in a neighborhood of
the unit circle. The characterization for the existence of strictly
stationary ARMA$(p,q)$ processes now reads as follows.

\begin{theorem} {\rm [Strict ARMA$(p,q)$ processes]} \label{thm-5}
\\
Let $m,d, p\in \bN$, $q\in \bN_0$, and let $(Z_t)_{t\in \bZ}$ be an
i.i.d. sequence of $\bC^d$-valued random vectors. Let $\Psi_1,
\ldots, \Psi_p \in \bC^{m\times m}$ and $\Theta_0, \ldots, \Theta_q
\in \bC^{m\times d}$ be complex-valued matrices, and define the
characteristic polynomials as in \eqref{eq-char}. Define the linear
subspace
$$K := \{ a \in \bC^d : \mbox{the distribution of}\; a^* Z_0 \; \mbox{is degenerate to a Dirac
measure}\}$$ of $\bC^d$, denote by $K^\perp$ its orthogonal
complement in $\bC^d$, and let $s:= \dim K^\perp$ the vector space
dimension of $K^\perp$. Let $U\in \bC^{d\times d}$ be  unitary such
that $U \, K^\perp = \bC^s \times \{0_{d-s}\}$ and $U\, K =
\{0_s\}\times \bC^{d-s}$, and define the $\bC^{m\times d}$-valued
rational function $M(z)$ by
\begin{equation}
 z \mapsto M(z)  :=  P^{-1}(z)  Q(z) U^* \left(
\begin{array}{ll} \mbox{\rm Id}_s & 0_{s,d-s}
\\ 0_{d-s,s} & 0_{d-s,d-s} \end{array} \right) .
 \label{meromorphic}
\end{equation}
Then there is a constant $u\in \bC^{d-s}$ and a $\bC^{s}$-valued
i.i.d. sequence $(w_t)_{t\in \bZ}$ such that
\begin{equation} \label{eq-uw}
U Z_t = \left( \begin{array}{c} w_t \\ u \end{array} \right)\quad
\mbox{a.s.} \quad \forall\; t\in \bZ, \end{equation} and the
distribution of $b^* w_0$ is not degenerate to a Dirac measure for
any $b\in \bC^s\setminus \{0\}$. Further, a strictly stationary
solution to the ARMA$(p,q)$ equation~\eqref{eqpq} exists if and only
if the following statements (i)---(iii) hold:
\begin{enumerate}
\addtolength{\itemsep}{-1.5ex}
\item[(i)]
All singularities on the unit circle of the meromorphic function
$M(z)$ are removable.
\item[(ii)] If $M(z) = \sum_{j=-\infty}^\infty M_j z^j$ denotes the
Laurent expansion of $M$ in a neighbourhood of the unit circle, then
\begin{equation} \label{eq-logfinite}
\bE \log^+ \| M_j UZ_0\| < \infty \quad \forall\; j \in \{ mp + q -
p +1 , \ldots, mp+q\} \cup \{-p,\ldots, -1\}.
\end{equation}
\item[(iii)] There exist $v\in \bC^s$ and $g\in \bC^m$ such that
 $g$ is a solution to the
linear equation
 \begin{equation} \label{eq-g}
P(1) g = Q(1) U^* (v^T, u^T)^T.
\end{equation}
\end{enumerate}
Further, if (i) above holds, then condition (ii) can be replaced by
\begin{enumerate}
\item[(ii')] If $M(z) = \sum_{j=-\infty}^\infty M_j z^j$ denotes the
Laurent expansion of $M$ in a neighbourhood of the unit circle, then
$\sum_{j=-\infty}^\infty M_j U Z_{t-j}$ converges almost surely
absolutely for every $t\in \bZ$,
\end{enumerate}
and condition (iii) can be replaced by
\begin{enumerate}
\item[(iii')] For all $v\in \bC^s$ there exists a solution $g=g(v)$ to
the linear equation \eqref{eq-g}.
\end{enumerate}
If the conditions (i)--(iii) given above are satisfied, then a
strictly stationary solution $Y$ of the ARMA$(p,q)$ equation
\eqref{eqpq} is given by
\begin{equation} \label{eq-Y}
Y_t = g + \sum_{j=-\infty}^\infty  M_j (UZ_{t-j} - (v^T, u^T)^T),
\quad t\in \bZ,\end{equation} the series converging almost surely
absolutely. Further, provided that the underlying probability space
is rich enough to support a random variable which is uniformly
distributed on $[0,1)$ and independent of $(Z_t)_{t\in \bZ}$, the
solution given by \eqref{eq-Y} is the unique strictly stationary
solution of \eqref{eqpq} if and only if $\det P(z) \neq 0 $ for all
$z$ on the unit circle.
\end{theorem}

Special cases of Theorem~\ref{thm-5} are treated in
Remarks \ref{rem-2a}, \ref{rem-2c} and Corollary \ref{cor-3}. Observe that for $m=1$,
Theorem~\ref{thm-5} reduces to the corresponding result in Brockwell
and Lindner~\cite{BL2}. Also observe that condition (iii) of
Theorem~\ref{thm-5} is not implied by condition (i), which can be
seen e.g. by allowing a deterministic noise sequence $(Z_t)_{t\in
\bZ}$, in which case $M(z) \equiv 0$.
 The proof of
Theorem~\ref{thm-5} will be given in Section~\ref{S6} and will make
use of both Theorem~\ref{thm-main} and Theorem~\ref{thm-4} given
below. The latter is the corresponding characterization for the
existence of weakly stationary solutions of ARMA$(p,q)$ equations,
expressed in terms of the characteristic polynomials $P(z)$ and
$Q(z)$. That $\det P(z) \neq 0$ for all $z$ on the unit circle
together with $\mathbb{E} (Z_0) = 0$ is sufficient for
the existence of weakly stationary solutions is well known, but that
the conditions given below are necessary and sufficient in higher
dimensions seems not to have appeared in the literature so far. The
proof of Theorem~\ref{thm-4}, which is similar to the proof in the
one-dimensional case, will be given in Section~\ref{S4}.

\begin{theorem} {\rm [Weak ARMA$(p,q)$ processes]} \label{thm-4} \\
Let $m,d,p\in \bN$, $q\in \bN_0$, and let $(Z_t)_{t\in \mathbb{Z}}$
be a weak white noise sequence in $\bC^d$ with expectation $\bE Z_0$
and covariance matrix $\Sigma$.  Let $\Psi_1, \ldots, \Psi_p \in
\bC^{m\times m}$ and $\Theta_0, \ldots, \Theta_q \in \bC^{m\times
d}$, and define the matrix polynomials $P(z)$ and $Q(z)$ by
\eqref{eq-char}. Let $U\in \bC^{d\times d}$ be  unitary
such that $U \Sigma U^* = \begin{pmatrix} D & 0_{s,d-s} \\
0_{d-s,s} & 0_{d-s,d-s} \end{pmatrix}$, where $D$ is a real
$(s\times s)$-diagonal matrix with the strictly positive eigenvalues
of $\Sigma$ on its diagonal for some $s\in \{0,\ldots, d\}$.  (The
matrix $U$ exists since $\Sigma$ is positive semidefinite). Then the
ARMA$(p,q)$ equation \eqref{eqpq} admits a weakly stationary
solution $(Y_t)_{t\in \bZ}$ if and only if the $\bC^{m\times
d}$-valued rational function
$$z \mapsto M(z) := P^{-1} (z) Q(z) U^* \left( \begin{array}{ll}
\mbox{\rm Id}_s  & 0_{s,d-s} \\ 0_{d-s,s} & 0_{d-s,d-s} \end{array}
\right)$$ has only removable singularities on the unit circle and if
there is some $g\in \bC^m$ such that \begin{equation} \label{eq-g2}
P(1)\, g = Q(1) \, \bE Z_0.\end{equation} In that case, a weakly
stationary solution of \eqref{eqpq} is given by
\begin{equation} \label{eq-weakly2} Y_t = g +
\sum_{j=-\infty}^\infty M_j  \, U(Z_{t-j}- \bE Z_0),\quad t\in \bZ,
\end{equation}
where $M(z) = \sum_{j=-\infty}^\infty M_j z^j$ is the Laurent
expansion of $M(z)$ in a neighbourhood of the unit circle, which
converges absolutely there.
\end{theorem}

It is easy to see that if $\Sigma$ in the theorem above is
invertible, then the condition that all singularities of $M(z)$ on
the unit circle are removable is equivalent to the condition that
all singularities of $P^{-1}(z) Q(z)$ on the unit circle are
removable.

\section{Proof of Theorem~\ref{thm-main}} \label{S5}
\setcounter{equation}{0}

 In this section we give the proof of
Theorem~\ref{thm-main}. In Section~\ref{S5a} we show that the
conditions (i) --- (iii) are necessary. The suffiency of the
conditions is proven in Section~\ref{S-sufficient}, while the
uniqueness assertion is established in Section~\ref{S-uniqueness}.

\subsection{The necessity of the conditions} \label{S5a}
Assume that $(Y_t)_{t\in \mathbb{Z}}$ is a strictly stationary
solution of equation (\ref{eq1q}). As observed before
Theorem~\ref{thm-main}, this implies that each of the equations
\eqref{block1q} admits a strictly stationary solution, where
$X_t^{(h)}$ is defined as in \eqref{eq-Xt}. Equation~\eqref{block1q}
is itself an ARMA$(1,q)$ equation with i.i.d. noise, so that for
proving  (i) -- (iii) we may assume that $H=1$, that $S = {\rm
Id}_m$ and that $\Phi := \Psi_1$ is an $m\times m$ Jordan block
corresponding to an eigenvalue $\lambda$. Hence we assume throughout
Section~\ref{S5a} that
\begin{equation} \label{eq1q2}
Y_t - \Phi Y_{t-1} = \sum_{k=0}^q \Theta_k Z_{t-k}, \quad t\in \bZ,
\end{equation}
has a strictly stationary solution with $\Phi\in \bC^{m\times m}$ of
the form \eqref{eq-Jordanblock2}, and we have to show that this
implies (i) if $|\lambda| \neq 0, 1$, (ii) if $|\lambda|=1$ but
$\lambda\neq 1$, and (iii) if $\lambda=1$. Before we do this in the
next subsections, we observe that iterating the ARMA$(1,q)$ equation
(\ref{eq1q2}) gives for $n\ge q$
\begin{eqnarray}
Y_t
&=&  \sum_{j=0}^{q-1}\Phi^j\left(\sum_{k=0}^j\Phi^{-k}\Theta_k\right)Z_{t-j}+\sum_{j=q}^{n-1}\Phi^j\left(\sum_{k=0}^q\Phi^{-k}\Theta_k\right)Z_{t-j}\notag\\
&& +
\sum_{j=0}^{q-1}\Phi^{n+j}\left(\sum_{k=j+1}^q\Phi^{-k}\Theta_k\right)Z_{t-(n+j)}+\Phi^nY_{t-n}.\label{it1}
\end{eqnarray}

\subsubsection{The case $|\lambda| \in (0,1)$.}\label{S-3-1-1}
Suppose that $|\lambda| \in (0,1)$ and let $\varepsilon  \in
(0,|\lambda|)$.
 Then there are constants $C, C'\geq 1$ such
that
\begin{eqnarray*}
\left\Vert\Phi^{-j}\right\Vert \le C \cdot|\lambda|^{-j}\cdot j^m
\leq (C') (|\lambda|-\varepsilon)^{-j} \quad\mbox{for all
}j\in\mathds{N},
\end{eqnarray*}
as a consequence of Theorem 11.1.1 in \cite{GolubVanLoan}. Hence, we
have for all $j\in \bN_0$ and $t\in \bZ$
\begin{eqnarray}
\left\Vert \left(\sum_{k=0}^q\Phi^{-k}\Theta_k\right)Z_{t-j} \right\Vert
%&\le& \left\Vert\Phi^{-j}\right\Vert\cdot\left\Vert \Phi^j\left(\sum_{k=0}^q\Phi^{-k}\Theta_k\right)Z_{t-j}\right\Vert\notag\\
&\le& {C'}(|\lambda|-\varepsilon)^{-j}\left\Vert
\Phi^{j}\left(\sum_{k=0}^q\Phi^{-k}\Theta_k\right)Z_{t-j}\right\Vert.\label{golub}
\end{eqnarray}
Now, since $\lim_{n\to\infty}\Phi^n=0$ and since  $(Y_t)_{t\in\bZ}$
and $(Z_t)_{t\in\bZ}$ are strictly stationary,  an application of
Slutsky's lemma to equation (\ref{it1}) shows that
\begin{eqnarray}
Y_t &=&
\sum_{j=0}^{q-1}\Phi^{j}\left(\sum_{k=0}^j\Phi^{-k}\Theta_k\right)Z_{t-j}
+\Prob\mbox{-}\lim_{n\to\infty}\sum_{j=q}^{n-1}\Phi^{j}\left(\sum_{k=0}^q\Phi^{-k}\Theta_k\right)Z_{t-j}.
\label{eq-uniqueness1}
\end{eqnarray}
Hence the limit on the right hand side exists and, as a sum with
independent summands, it converges almost surely. Thus it follows
from equation (\ref{golub}) and the Borel-Cantelli lemma that
\begin{eqnarray*}
\lefteqn{\sum_{j=q}^{\infty}\Prob\left(\left\Vert\sum_{k=0}^q\Phi^{-k}\Theta_kZ_{0}
\right\Vert> {C'}(|\lambda|-\varepsilon)^{-j}  \right)}\\
%&=& \sum_{j=q}^{\infty}\Prob\left(C'(|\mu_1|-\varepsilon)^j\left\Vert\sum_{k=0}^q\Phi^{-k}\Theta_kZ_{-j} \right\Vert>1  \right)\\
&\le&
\sum_{j=q}^{\infty}\Prob\left(\left\Vert\Phi^j\left(\sum_{k=0}^q\Phi^{-k}\Theta_k\right)Z_{-j}\right\Vert>1\right)<\infty,
\end{eqnarray*}
and hence $\mathds{E}\left( \log^+
\left\Vert\left(\sum_{k=0}^q\Phi^{-k}\Theta_k\right)Z_{0}\right\Vert\right)<\infty.$
Obviously, this is equivalent to condition (i).

\subsubsection{The case $|\lambda|>1$.} \label{S-3-1-2}
Suppose that $|\lambda| > 1$. Multiplying equation (\ref{it1}) by
$\Phi^{-n}$ gives for $n\ge q$
\begin{eqnarray*}
\Phi^{-n}Y_t
&=&  \sum_{j=0}^{q-1}\Phi^{-(n-j)}\left(\sum_{k=0}^j\Phi^{-k}\Theta_k\right)Z_{t-j}+\sum_{j=1}^{n-q}\Phi^{-j}\left(\sum_{k=0}^q\Phi^{-k}\Theta_k\right)Z_{t-n+j}\notag\\
&& +
\sum_{j=0}^{q-1}\Phi^{j}\left(\sum_{k=j+1}^q\Phi^{-k}\Theta_k\right)Z_{t-(n+j)}+Y_{t-n}.
\end{eqnarray*}
Defining $\tilde{\Phi}:=\Phi^{-1}$, and substituting $u=t-n$ yields
\begin{eqnarray}
Y_u
&=&  - \sum_{j=0}^{q-1}\tilde{\Phi}^{-j}\left(\sum_{k=j+1}^q\Phi^{-k}\Theta_k\right)Z_{u-j}-\sum_{j=1}^{n-q}\tilde{\Phi}^{j}\left(\sum_{k=0}^q\Phi^{-k}\Theta_k\right)Z_{u+j}\notag\\
&&
-\sum_{j=0}^{q-1}\tilde{\Phi}^{n-j}\left(\sum_{k=0}^j\Phi^{-k}\Theta_k\right)Z_{u+n-j}+\tilde{\Phi}^{n}Y_{u+n}
.\label{itminus}
\end{eqnarray}
Letting $n\to\infty$ then gives condition (i) with the same
arguments as in the case $|\lambda| \in (0,1)$.

\subsubsection{The case $|\lambda|=1$ and symmetric noise $(Z_t)$.}
\label{S-3-1-3}

Suppose that $Z_0$ is symmetric and that $|\lambda|=1$. Denoting
$$J_1 := \Phi - \lambda \, {\rm Id}_m \quad \mbox{and} \quad J_l := J_1^l
\quad \mbox{for}\quad j\in \bN_0,$$ we have
\begin{eqnarray*}
\Phi^{j} &=& \sum_{l=0}^{m-1}\binom{j}{l}\lambda^{j-l}J_l,\quad
j\in\mathds{N}_0,
\end{eqnarray*}
since $J_l=0$ for $l\geq m$ and $\binom{j}{l} = 0$ for $l>j$.
Further, since for $l\in \{0,\ldots, m-1\}$ we have
\begin{eqnarray*}
J_l=\left(e_{l+1},e_{l+2},...,e_{m},0_{m},...,0_{m}\right)\in\mathbb{C}^{m\times
m},
\end{eqnarray*}
with unit vectors $e_{l+1},...,e_m$ in $\mathbb{C}^{m}$, it is easy
to see that for $i=1,...,m$ the $i^{th}$ row of the matrix $\Phi^j$
is given by
\begin{equation}
e_i^T\Phi^{j} = \sum_{l=0}^{m-1}\binom{j}{l}\lambda^{j-l}e_i^TJ_l =
\sum_{l=0}^{i-1}\binom{j}{l}\lambda^{j-l}e_{i-l}^T, \quad j\in
\bN_0.\label{phij}
\end{equation}
It follows from equations (\ref{it1}) and (\ref{phij}) that for
$n\ge q$ and $t\in \bZ$,
\begin{eqnarray}
e_i^TY_t
&=&  \sum_{j=0}^{q-1}\left(\sum_{l=0}^{i-1}\binom{j}{l}\lambda^{j-l}e_{i-l}^T\right)\left(\sum_{k=0}^j\Phi^{-k}\Theta_k\right)Z_{t-j}\notag\\
&&+\sum_{j=q}^{n-1}\left(\sum_{l=0}^{i-1}\binom{j}{l}\lambda^{j-l}e_{i-l}^T\right)\left(\sum_{k=0}^q\Phi^{-k}\Theta_k\right)Z_{t-j}\notag\\
&& + \sum_{j=0}^{q-1}\left(\sum_{l=0}^{i-1}\binom{n+j}{l}\lambda^{n+j-l}e_{i-l}^T\right)\left(\sum_{k=j+1}^q\Phi^{-k}\Theta_k\right)Z_{t-(n+j)}\notag\\
&&+
\sum_{l=0}^{i-1}\binom{n}{l}\lambda^{n-l}e_{i-l}^TY_{t-n}.\label{it2}
\end{eqnarray}
We claim that
\begin{equation} \label{bed2c}
e_i^T \sum_{k=0}^q \Phi^{-k} \Theta_k Z_t = 0 \;\; \mbox{a.s.}\quad
\forall\; i \in \{1,\ldots, m\} \quad \forall\; t\in \bZ,
\end{equation}
  which clearly gives conditions (ii) and (iii), respectively, with
  $\alpha=\alpha_1=0_m$. Equation \eqref{bed2c} will be proved by induction on
$i=1,\ldots, m$.
  We start with $i=1$.
From equation (\ref{it2}) we know that for $n\ge q$
\begin{eqnarray}
& &{e_1^TY_t-\lambda^{n}e_{1}^TY_{t-n}-\sum_{j=0}^{q-1}\lambda^{j}e_{1}^T\left(\sum_{k=0}^j\Phi^{-k}\Theta_k\right)Z_{t-j}-\sum_{j=0}^{q-1}\lambda^{n+j}e_{1}^T\left(\sum_{k=j+1}^q\Phi^{-k}\Theta_k\right)Z_{t-(n+j)}}\notag\\
& = &
\sum_{j=q}^{n-1}\lambda^{j}e_{1}^T\left(\sum_{k=0}^q\Phi^{-k}\Theta_k\right)Z_{t-j}.\label{IAc}
\end{eqnarray}
Due to the stationarity of $(Y_t)_{t\in\mathds{Z}}$ and
$(Z_t)_{t\in\mathds{Z}}$, there exists a constant $K_1>0$ such that
\begin{eqnarray*}
&&\Prob\left(\left|
e_1^TY_t-\lambda^{n}e_{1}^TY_{t-n}-\sum_{j=0}^{q-1}\lambda^{j}e_{1}^T\left(\sum_{k=0}^j\Phi^{-k}\Theta_k\right)Z_{t-j}
\right.\right.\notag\\&&\quad\quad\quad\quad\quad\quad\quad\quad\left.\left.-\sum_{j=0}^{q-1}\lambda^{n+j}e_{1}^T\left(\sum_{k=j+1}^q\Phi^{-k}\Theta_k\right)Z_{t-(n+j)}\right|<K_1\right)\ge\frac{1}{2}\quad\forall
n\geq q.
\end{eqnarray*}
By \eqref{IAc} this implies
\begin{eqnarray}
&&\Prob\left(\left|\sum_{j=q}^{n-1}\lambda^{j}e_{1}^T\left(\sum_{k=0}^q\Phi^{-k}\Theta_k\right)Z_{t-j}\right|<K_1\right)\ge\frac{1}{2}\quad\forall
n\geq q. \label{eq-symmetrise1}
\end{eqnarray}
Therefore
$\left|\sum_{j=q}^{n-1}\lambda^{j}e_{1}^T\left(\sum_{k=0}^q\Phi^{-k}\Theta_k\right)Z_{t-j}\right|$
does not converge in probability to $+\infty$ as $n\to\infty$. Since
this is a sum of independent and symmetric terms, this implies that
it converges almost surely (see Kallenberg \cite{Kallenberg},
Theorem 4.17), and the Borel-Cantelli lemma then shows that
\begin{eqnarray*}
e_1^T\left(\sum_{k=0}^q\Phi^{-k}\Theta_k\right)Z_{t}=0,\quad t\in
\bZ,
\end{eqnarray*}
which is \eqref{bed2c} for $i=1$. With this condition, equation
(\ref{IAc}) simplifies for $t=0$ and $n\geq q$ to
%\begin{eqnarray*}
%e_1^TY_t=\sum_{j=0}^{q-1}\lambda^{j}e_{1}^T\left(\sum_{k=0}^j\Phi^{-k}\Theta_k\right)Z_{t-j}+\sum_{j=0}^{q-1}\lambda^{n+j}e_{1}^T\left(\sum_{k=j+1}^q\Phi^{-k}\Theta_k\right)Z_{t-(n+j)}+\lambda^{n}e_{1}^TY_{t-n}.
%\end{eqnarray*}
%With $t=0$ we get for $n\ge q$
\begin{eqnarray*}
e_1^TY_0-\lambda^{n}e_{1}^TY_{-n}=\sum_{j=0}^{q-1}\lambda^{j}e_{1}^T\left(\sum_{k=0}^j\Phi^{-k}\Theta_k\right)Z_{-j}+\sum_{j=0}^{q-1}\lambda^{n+j}e_{1}^T\left(\sum_{k=j+1}^q\Phi^{-k}\Theta_k\right)Z_{-(n+j)}.
\end{eqnarray*}
Now setting $t:=-n$ in the above equation, multiplying it with
$\lambda^t=\lambda^{-n}$ and recalling that $e_1^T \Phi^j =
\lambda^j e_1^T$ by \eqref{phij} yields for $t\le-q$
\begin{eqnarray*}
e_1^TY_t=-\sum_{j=0}^{q-1}e_{1}^T \Phi^j
\left(\sum_{k=j+1}^q\Phi^{-k}\Theta_k\right)Z_{t-j}+\lambda^te_1^T\left(Y_0-\sum_{j=0}^{q-1}\Phi^{j}\left(\sum_{k=0}^j\Phi^{-k}\Theta_k\right)Z_{-j}\right).
\end{eqnarray*}
For the induction step let $i\in \{2,\ldots, m\}$ and assume that
\begin{eqnarray}
e_r^T\left(\sum_{k=0}^q\Phi^{-k}\Theta_k\right)Z_{t}=0\;\;
\mbox{a.s.},\quad r\in\{1,...,i-1\},\; \; t\in \bZ,\label{IVc1}
\end{eqnarray}
together with
\begin{eqnarray}
e_r^TY_t =
-e_{r}^T\sum_{j=0}^{q-1}\Phi^j\left(\sum_{k=j+1}^q\Phi^{-k}\Theta_k\right)Z_{t-j}+\begin{cases}
\displaystyle 0, &   r\in \{1,\ldots, i-2\}, \; t \leq -rq,\\
    \displaystyle \lambda^te_r^TV_r, & r=i-1,\; t \leq -rq,
    \end{cases}\label{IVc2}
\end{eqnarray}
where
$$V_r:=\lambda^{(r-1)q} \left( Y_{-(r-1)q}
-\sum_{j=0}^{q-1}\Phi^j\left(\sum_{k=0}^j\Phi^{-k}\Theta_k\right)Z_{-j-(r-1)q}\right),
\quad r \in \{1,\ldots, m\}.$$  We are going to show that this
implies
\begin{eqnarray}
e_i^T\left(\sum_{k=0}^q\Phi^{-k}\Theta_k\right)Z_{t}=0\;
\;\mbox{a.s.},\quad t \in \bZ,\label{ISc}
\end{eqnarray}
and
\begin{eqnarray}
e_i^TY_t &=&
-e_{i}^T\sum_{j=0}^{q-1}\Phi^j\left(\sum_{k=j+1}^q\Phi^{-k}\Theta_k\right)Z_{t-j}+\lambda^te_i^TV_i
\; \; \mbox{a.s.}, \quad t \leq -iq,\label{IScBeh2}
\end{eqnarray}
together with
\begin{eqnarray}
e_{i-1}^T V_{i-1}=0.\label{IScBeh3}
\end{eqnarray}
This will then imply \eqref{bed2c}. For doing that, in a first step
we are going to prove the following:

\begin{lemma}\label{lemma}
Let $i\in \{2,\ldots, m\}$ and assume (\ref{IVc1}) and (\ref{IVc2}).
Then it holds for $t\leq -(i-1)q$ and $n\geq q$,
\begin{eqnarray}
e_i^TY_t -\lambda^ne_i^TY_{t-n}
&=&  \sum_{j=0}^{q-1}e_{i}^T\Phi^j\left(\sum_{k=0}^j\Phi^{-k}\Theta_k\right)Z_{t-j}+\sum_{j=q}^{n-1}\lambda^je_i^T\left(\sum_{k=0}^q\Phi^{-k}\Theta_k\right)Z_{t-j}\notag\\
&&+\lambda^n\sum_{j=0}^{q-1}e_{i}^T\Phi^j\left(\sum_{k=j+1}^q\Phi^{-k}\Theta_k\right)Z_{t-(n+j)}+n\lambda^{t-1}e_{i-1}^TV_{i-1},\label{ISc7}
\end{eqnarray}
\end{lemma}

\begin{proof} Let $t\leq -(i-1)q$ and $n\geq q$. Using  \eqref{IVc2} and
\eqref{phij}, the last summand of \eqref{it2} can be written as
\begin{eqnarray*}
\lefteqn{\sum_{l=0}^{i-1}\binom{n}{l}\lambda^{n-l}e_{i-l}^TY_{t-n}}
\notag\\
% &=& \lambda^ne_i^TY_{t-n}+\sum_{l=1}^{i-1}\binom{n}{l}\lambda^{n-l}e_{i-l}^TY_{t-n}\\
&=&
\lambda^ne_i^TY_{t-n}+\sum_{r=1}^{i-1}\binom{n}{i-r}\lambda^{n-(i-r)}e_{r}^TY_{t-n},
\notag \\
%&&\sum_{l=0}^{i-1}\binom{n}{l}\lambda^{n-l}e_{i-l}^TY_{t-n}\notag\\
&=& \lambda^ne_i^TY_{t-n}-\sum_{j=0}^{q-1}\left(\sum_{r=1}^{i-1}\sum_{l=0}^{r-1}\binom{j}{l}\binom{n}{i-r}\lambda^{n-(i-r)}\lambda^{j-l}e_{r-l}^T\right)\left(\sum_{k=j+1}^q\Phi^{-k}\Theta_k\right)Z_{t-(n+j)}\notag\\
&&+n\lambda^{t-1}e_{i-1}^T V_{i-1}\\
&=& \lambda^ne_i^TY_{t-n}-\sum_{j=0}^{q-1}\left( \sum_{s=1}^{i-1}\binom{n+j}{s}\lambda^{n+j-s}e_{i-s}^T\right)\left(\sum_{k=j+1}^q\Phi^{-k}\Theta_k\right)Z_{t-(n+j)}\notag\\
&&+\lambda^n\sum_{j=0}^{q-1}\left(\sum_{s=1}^{i-1}\binom{j}{s}\lambda^{j-s}e^T_{i-s}\right)\left(\sum_{k=j+1}^q\Phi^{-k}\Theta_k\right)Z_{t-(n+j)}+n\lambda^{t-1}e_{i-1}^TV_{i-1},
\end{eqnarray*}
where we substituted $s:= i-r+l$ and $p:= s-l$ and used
Vandermonde's identity $\sum_{p=1}^s \binom{j}{s-p} \binom{n}{p} =
\binom{n+j}{s} - \binom{j}{s}$ in the last equation. Inserting this
back into equation (\ref{it2}) and using \eqref{IVc1}, we get for
$t\leq -(i-1)q$ and $n\geq q$
\begin{eqnarray*}
\lefteqn{e_i^TY_t -\lambda^ne_i^TY_{t-n}}\\
&=&
\sum_{j=0}^{q-1}\left(\sum_{l=0}^{i-1}\binom{j}{l}\lambda^{j-l}e_{i-l}^T\right)\left(\sum_{k=0}^j\Phi^{-k}\Theta_k\right)Z_{t-j}
\\
& &
+\sum_{j=q}^{n-1}\lambda^je_i^T\left(\sum_{k=0}^q\Phi^{-k}\Theta_k\right)Z_{t-j}
 +\sum_{j=0}^{q-1} \lambda^{n+j}e_{i}^T \left(\sum_{k=j+1}^q\Phi^{-k}\Theta_k\right)Z_{t-(n+j)}\notag\\
&&+\lambda^n\sum_{j=0}^{q-1}\left(\sum_{s=1}^{i-1}\binom{j}{s}\lambda^{j-s}e^T_{i-s}\right)\left(\sum_{k=j+1}^q\Phi^{-k}\Theta_k\right)Z_{t-(n+j)}\notag\\
&&+n\lambda^{t-1}e_{i-1}^TV_{i-1}.\label{ISc4}
\end{eqnarray*}
An application of (\ref{phij}) then shows (\ref{ISc7}), completing
the proof of the lemma.\end{proof}

To continue with the induction step, we first show that
(\ref{IScBeh3}) holds true. Dividing (\ref{ISc7}) by $n$ and letting
$n\to\infty$, the strict stationarity of $(Y_t)_{t\in \bZ}$ and
$(Z_t)_{t\in \bZ}$ imply that for $t\leq -(i-1)q$,
$$n^{-1}\sum_{j=q}^{n-1}\lambda^je_i^T\left(\sum_{k=0}^q\Phi^{-k}\Theta_k\right)Z_{t-j}$$
converges in probability to $-\lambda^{t-1} e_{i-1}^T V_{i-1}$. On
the other hand, this limit in probability must be clearly measurable
with respect to the tail-$\sigma$-algebra
$\cap_{k\in\mathds{N}}\sigma(\cup_{l\ge k}\sigma(Z_{t-l}))$, which
by Kolmogorov's zero-one law is $\Prob$-trivial. Hence this
probability limit must be constant, and because of the assumed
symmetry of $Z_0$ it must be symmetric, hence is equal to 0, i.e.
\begin{eqnarray*}
e_{i-1}^TV_{i-1}=0 \; \;\mbox{a.s.},
\end{eqnarray*}
which is (\ref{IScBeh3}). Using this, we get from Lemma \ref{lemma}
that
\begin{eqnarray}
& & {e_i^TY_t
-\lambda^ne_i^TY_{t-n}-\sum_{j=0}^{q-1}e_{i}^T\Phi^j\left(\sum_{k=0}^j\Phi^{-k}\Theta_k\right)Z_{t-j}
-\lambda^n\sum_{j=0}^{q-1}e_{i}^T\Phi^j\left(\sum_{k=j+1}^q\Phi^{-k}\Theta_k\right)Z_{t-(n+j)}}\notag\\
&&=\sum_{j=q}^{n-1}\lambda^je_i^T\left(\sum_{k=0}^q\Phi^{-k}\Theta_k\right)Z_{t-j},
\quad t\leq -(i-1)q.\label{ISc9}
\end{eqnarray}
Again due to the stationarity of $ (Y_t)_{t\in\mathds{Z}}$ and $
(Z_t)_{t\in\mathds{Z}}$ there exists a constant $K_2>0$ such that
\begin{eqnarray*}
&&\Prob\left(\left|e_i^TY_t -\lambda^ne_i^TY_{t-n}-\sum_{j=0}^{q-1}e_{i}^T\Phi^j\left(\sum_{k=0}^j\Phi^{-k}\Theta_k\right)Z_{t-j}\right.\right.\\
&&\left.\left.\quad\quad-\lambda^n\sum_{j=0}^{q-1}e_{i}^T\Phi^j\left(\sum_{k=j+1}^q\Phi^{-k}\Theta_k\right)Z_{t-(n+j)}\right|<K_2\right)\ge\frac{1}{2}\quad\forall\;
n\geq q,
\end{eqnarray*}
so that
\begin{eqnarray*}
&&\Prob\left(\left|
\sum_{j=q}^{n-1}\lambda^je_i^T\left(\sum_{k=0}^q\Phi^{-k}\Theta_k\right)Z_{t-j}\right|<K_2\right)\ge\frac{1}{2}\quad\forall\;
n\geq q, \; \; \; t\leq -(i-1)q.
\end{eqnarray*}
Therefore
$\left|\sum_{j=q}^{n-1}\lambda^je_i^T\left(\sum_{k=0}^q\Phi^{-k}\Theta_k\right)Z_{t-j}\right|$
does not converge in probability to $+\infty$ as $n\to\infty$. Since
this is a sum of independent and symmetric terms, this implies that
it converges almost surely (see Kallenberg \cite{Kallenberg},
Theorem 4.17), and the Borel-Cantelli lemma then shows that $
e_i^T\left(\sum_{k=0}^q\Phi^{-k}\Theta_k\right)Z_{t}=0$ a.s. for
$t\leq -(i-1)q$ and hence for all $t\in \bZ$, which is (\ref{ISc}).
Equation (\ref{ISc9}) now simplifies for $t=-(i-1)q$ and $n\geq q$
to
\begin{eqnarray*}
\lefteqn{e_i^TY_{-(i-1)q} -\lambda^n e_i^T Y_{-(i-1)q-n}} \\
&  =
&\sum_{j=0}^{q-1}e_{i}^T\Phi^j\left(\sum_{k=0}^j\Phi^{-k}\Theta_k\right)Z_{-(i-1)q-j}+
\lambda^n\sum_{j=0}^{q-1}e_{i}^T\Phi^j\left(\sum_{k=j+1}^q\Phi^{-k}\Theta_k\right)Z_{-(i-1)q-n-j}.
\end{eqnarray*}
Multiplying this equation by $\lambda^{-n}$ and denoting
$t:=-(i-1)q-n$, it follows that for $t\le-iq$ it holds
\begin{eqnarray*}
e_i^TY_t &=&
-\sum_{j=0}^{q-1}e_{i}^T\Phi^j\left(\sum_{k=j+1}^q\Phi^{-k}\Theta_k\right)Z_{t-j}
\\ & & +\lambda^{t+(i-1)q}
e_i^T\left(Y_{-(i-1)q}-\sum_{j=0}^{q-1}\Phi^j\left(\sum_{k=0}^j\Phi^{-k}\Theta_k\right)Z_{-j-(i-1)q}\right)\\
&=&
-\sum_{j=0}^{q-1}e_{i}^T\Phi^j\left(\sum_{k=j+1}^q\Phi^{-k}\Theta_k\right)Z_{t-j}+\lambda^te_i^TV_i,
\end{eqnarray*}
which is equation (\ref{IScBeh2}). This completes the proof  of the
induction step and hence of \eqref{bed2c}. It follows that
conditions (ii) and (iii), respectively, hold with $\alpha_1=0$ if
$|\lambda|=1$ and $Z_0$ is symmetric.

\subsubsection{The case $|\lambda|=1$ and not necessarily symmetric noise $(Z_t)$.}
\label{S-3-1-4}

As in Section~\ref{S-3-1-3}, assume that $|\lambda|=1$, but not
necessarily that $Z_0$ is symmetric.
 Let
$(Y_t', Z_t')_{t\in \bZ}$ be an independent copy of $(Y_t,Z_t)_{t\in
\bZ}$ and denote $\widetilde{Y}_t := Y_t - Y_t'$ and
$\widetilde{Z}_t := Z_t - Z_t'$. Then $(\widetilde{Y}_t)_{t\in \bZ}$
is a strictly stationary solution of $\widetilde{Y}_t - \Phi
\widetilde{Y}_{t-1} = \sum_{k=0}^q \Theta_k \widetilde{Z}_{t-k}$,
and $(\widetilde{Z}_t)_{t\in \bZ}$ is i.i.d. with $\widetilde{Z}_0$
being symmetric. It hence follows from Section~\ref{S-3-1-3} that
$$\left( \sum_{k=0}^q \Phi^{q-k} \Theta_k \right) Z_0 - \left(
\sum_{k=0}^q \Phi^{q-k} \Theta_k \right) Z_0' = \left( \sum_{k=0}^q
\Phi^{q-k} \Theta_k\right) \widetilde{Z}_0 = 0.$$ Since $Z_0$ and
$Z_0'$ are independent, this implies that there is a constant
$\alpha\in \bC^m$ such that $\sum_{k=0}^q \Phi^{q-k} \Theta_k Z_0 =
\alpha$ a.s., which is \eqref{bed2a}, hence condition (ii) if
$\lambda\neq 1$. To show condition (iii) in the case $\lambda=1$,
recall that the deviation of \eqref{eq-symmetrise1} in
Section~\ref{S-3-1-3} did not need the symmetry assumption on $Z_0$.
Hence by \eqref{eq-symmetrise1} there is some constant $K_1$ such
that $ \Prob(|\sum_{j=q}^{n-1} 1^{j}e_{1}^T \alpha |<K_1)\geq 1/2$
for all $n\geq q$, which clearly implies $e_1^T \alpha = 0$ and
hence condition (iii).

\subsection{The sufficiency of the conditions}\label{S-sufficient}
Suppose that conditions (i) --- (iii) are satisfied, and let
$X_t^{(h)}$, $t\in \bZ$, $h\in \{1,\ldots, H\}$, be defined by
\eqref{eq-solution1}. The fact that $X_{t}^{(h)}$ as defined in
\eqref{eq-solution1} converges a.s. for $|\lambda_h| \in (0,1)$ is
in complete analogy to the proof in the one-dimensional case treated
in Brockwell and Lindner~\cite{BL2}, but we give the short argument
for completeness: observe that there are constants $a, b
> 0$ such that $\|\Phi_{h}^j \| \leq a e^{-bj}$ for $j\in
\bN_0$. Hence for $b' \in (0,b)$ we can estimate
\begin{eqnarray*}
\lefteqn{\sum_{j=q}^\infty \mathbb{P} \left( \left\| \Phi_{h}^{j-q}
\sum_{k=0}^q \Phi_{h}^{q-k} I_{h} S^{-1}
\Theta_k Z_{t-j} \right\| > e^{-b' (j-q)} \right)}\\
&  \leq &  \sum_{j=q}^\infty \mathbb{P} \left( \log^+ \left(a
\left\| \sum_{k=0}^q \Phi_{h}^{q-k} I_{h} S^{-1} \Theta_k Z_{t-j}
\right\|\right)
> (b-b') (j-q) \right)
 < \infty,\end{eqnarray*}
the last inequality being due to the fact that $\left\| \sum_{k=0}^q
\Phi_{h}^{q-k} I_{h} S^{-1} \Theta_k Z_{t-j} \right\|$ has the same
distribution as $\left\| \sum_{k=0}^q \Phi_{h}^{q-k} I_{h} S^{-1}
\Theta_k Z_{0} \right\|$ and the latter has finite log-moment by
\eqref{bed1a}. The Borel--Cantelli lemma then shows that the event
$\{\| \Phi_{h}^{j-q} \sum_{k=0}^q \Phi_{h}^{q-k} I_{h} S^{-1}
\Theta_k Z_{t-j} \| > e^{-b' (j-q)} \; \mbox{for infinitely many
$j$}\}$ has probability zero, giving the almost sure absolute
convergence of the series in \eqref{eq-solution1}. The almost sure
absolute convergence of \eqref{eq-solution1} if $|\lambda_h|>1$ is
established similarly.

It is obvious that $((X_t^{(1)T}, \ldots, X_{t}^{(H)T})^T)_{t\in
\bZ}$ as defined in \eqref{eq-solution1} and hence $(Y_t)_{t\in
\bZ}$ defined by \eqref{eq-solution} is strictly stationary, so it
only remains to show that $(X_t^{(h)})_{t\in \bZ}$ solves
\eqref{block1q} for each $h\in \{1,\ldots, H\}$. For
$|\lambda_h|\neq 0,1$, this is an immediate consequence of
\eqref{eq-solution1}. For $|\lambda_h|=1$, we have by
\eqref{eq-solution1} and the definition of $f_h$ that
\begin{eqnarray*}
X_{t}^{(h)} - \Phi_{h} X_{t-1}^{(h)} & = & \alpha_h +
\sum_{j=0}^{q-1} \sum_{k=0}^j \Phi_{h}^{j-k} I_{h} S^{-1} \Theta_k
Z_{t-j} - \sum_{j=1}^{q} \sum_{k=0}^{j-1} \Phi_{h}^{j-k} I_{h}
S^{-1} \Theta_k Z_{t-j}\\
& = & \alpha_h + \sum_{j=0}^{q-1} I_{h} S^{-1} \Theta_j Z_{t-j} -
\sum_{k=0}^{q-1} \Phi_{h}^{q-k} I_{h} S^{-1} \Theta_k
Z_{t-q} \\
 &= &   I_{h} S^{-1} \sum_{j=0}^q \Theta_j Z_{t-j},
 \end{eqnarray*}
where the last equality follows from \eqref{bed2a}. Finally, if
$\lambda_h=0$, then $\Phi_h^j=0$ for $j\geq m$, implying that
$X_t^{(h)}$ defined by \eqref{eq-solution1} solves \eqref{block1q}
also in this case.

\subsection{The uniqueness of the solution} \label{S-uniqueness}

Suppose that $|\lambda_h|\neq 1$ for all $h\in \{1,\ldots, H\}$ and
let $(Y_t)_{t\in\bZ}$ be a strictly stationary solution of
\eqref{eq1q}. Then $(X_t^{(h)})_{t\in \bZ}$, as defined by
\eqref{eq-Xt}, is a strictly stationary solution of \eqref{block1q}
for each $h\in \{1,\ldots, H\}$. It then follows as in
Section~\ref{S-3-1-1} that by the equation corresponding to
\eqref{eq-uniqueness1}, $X_t^{(h)}$ is uniquely determined if
$|\lambda_h|\in (0,1)$. Similarly, $X_t^{(h)}$ is uniquely
determined if $|\lambda_h|>1$. The uniqueness of $X_{t}^{(h)}$ if
$\lambda_h=0$ follows from the equation corresponding to \eqref{it1}
with $n\geq m$, since then $\Phi_h^j = 0$ for $j\geq m$. We conclude
that $((X_t^{(1)T},\ldots, X_t^{(H)T})^T)_{t\in \bZ}$ is unique and
hence so is $(Y_t)_{t\in \bZ}$.

Now suppose that there is $h\in \{1,\ldots, H\}$ such that
$|\lambda_h|=1$. Let $U$ be a random variable which is uniformly
distributed on $[0,1)$ and independent of $(Z_t)_{t\in \bZ}$. Then
$(R_t)_{t\in \bZ}$, defined by $R_t := \lambda_h^t (0 ,\ldots 0 ,
e^{2\pi i U} )^T \in \bC^{r_{h+1}-r_h}$, is strictly stationary and
independent of $(Z_t)_{t\in \bZ}$ and satisfies $R_t - \Phi_h
R_{t-1} = 0$. Hence, if $(Y_t)_{t\in \bZ}$ is the strictly
stationary solution of \eqref{eq1q} specified by
\eqref{eq-solution1} and \eqref{eq-solution}, then
$$Y_t + S
(0_{r_2-r_1}^T, \ldots, 0_{r_h-r_{h-1}}^T, R_t^T,
0_{r_{h+2}-r_{h+1}}^T, \ldots, 0_{r_{H+1}-r_H}^T)^T, \quad t\in
\bZ,$$ is another strictly stationary solution of \eqref{eq1q},
violating uniqueness.

\section{Proof of Theorem~\ref{thm-4}} \label{S4}
\setcounter{equation}{0}

In this section we shall prove Theorem~\ref{thm-4}. Denote
$$R := U^* \begin{pmatrix} D^{1/2} & 0_{s,d-s} \\
0_{d-s,s} & 0_{d-s,d-s} \end{pmatrix} \quad \mbox{and} \quad W_t :=
\begin{pmatrix} D^{-1/2} & 0_{s,d-s}
\\ 0_{d-s,s} & 0_{d-s,d-s} \end{pmatrix} U (Z_t- \bE Z_0) , \quad t \in \bZ,$$
 where $D^{1/2}$ is the unique diagonal matrix with strictly
positive eigenvalues such that $(D^{1/2})^2 = D$. Then $(W_t)_{t\in
\bZ}$ is a white noise sequence in $\bC^d$ with expectation 0 and
covariance matrix $\begin{pmatrix} \mbox{\rm Id}_s & 0_{s,d-s} \\
0_{d-s,s} & 0_{d-s,d-s} \end{pmatrix}$. It is further clear that all
singularities of $M(z)$ on the unit circle are removable if and only
if all singularities of $M'(z):= P^{-1}(z) Q(z) R$ on the unit
circle are removable, and in that case, the Laurent expansions of
both $M(z)$ and $M'(z)$ converge almost surely absolutely in a
neighbourhood of the unit circle.

To see the sufficiency of the condition, suppose that \eqref{eq-g2}
has a solution $g$ and that $M(z)$ and hence $M'(z)$ have only
removable singularities on the unit circle. Define $Y=(Y_t)_{t\in
\bZ}$ by \eqref{eq-weakly2}, i.e.
$$Y_t = g + \sum_{j=-\infty}^\infty M_j \left( \begin{array}{ll}
D^{1/2} &  0_{s,d-s} \\ 0_{d-s,s} & 0_{d-s,d-s} \end{array} \right)
W_{t-j} = g + M'(B) W_t, \quad t\in \bZ.$$ The series converges
almost surely absolutely due to the exponential decrease of the
entries of $M_j$ as $|j|\to\infty$. Further,  $Y$ is clearly weakly
stationary, and since the last $(d-s)$ components of $U (Z_t - \bE
Z_0)$ vanish, having expectation zero and variance zero, it follows
that
$$R W_t = U^* \begin{pmatrix} \mbox{\rm Id}_s & 0_{s,d-s} \\ 0_{d-s,s} &
0_{d-s,d-s} \end{pmatrix} U(Z_t-\bE Z_0) = U^* U (Z_t-\bE Z_0) =
Z_t- \bE Z_0, \quad t\in \bZ.$$ We conclude that
$$P(B) (Y_t-g) = P(B) M'(B) W_t = P(B) P^{-1}(B) Q(B) R W_t  = Q(B) (Z_t-\bE Z_0), \quad t\in \bZ.$$
Since $P(1) g = Q(1) \bE Z_0$, this shows that $(Y_t)_{t\in \bZ}$ is
a weakly stationary solution of \eqref{eqpq}.

Conversely, suppose that $Y=(Y_t)_{t\in \bZ}$ is a weakly stationary
solution of \eqref{eqpq}. Taking expectations in \eqref{eqpq} yields
$P(1) \, \bE Y_0 = Q(1) \, \bE Z_0$, so that \eqref{eq-g2} has a
solution. The $\bC^{m\times m}$-valued spectral measure $\mu_Y$ of
$Y$ satisfies
$$P(e^{-i \omega}) \, d\mu_Y(\omega) \, P(e^{-i\omega})^* =
\frac{1}{2\pi} Q(e^{-i\omega}) \Sigma Q(e^{-i\omega})^* \, d\omega,
\quad \omega \in (-\pi, \pi].$$ It follows that, with the finite set
$N:= \{ \omega \in (-\pi,\pi]: P(e^{-i\omega}) = 0\}$,
$$d\mu_Y(\omega) = \frac{1}{2\pi} P^{-1}( e^{-i\omega}) Q(e^{-i\omega}) \Sigma
Q(e^{-i \omega})^* P^{-1} (e^{-i\omega})^* \, d\omega \quad
\mbox{on}\quad (-\pi,\pi] \setminus N.$$ Observing that $R R^*=
\Sigma$, it follows that the function $\omega \mapsto
M'(e^{-i\omega}) M'(e^{-i \omega})^*$ must be integrable on
$(-\pi,\pi] \setminus N$. Now assume that the matrix rational
function $M'$ has a non-removable singularity at $z_0$ with $|z_0| =
1$ in at least one matrix element. This must then be a pole of order
$r\geq 1$. Denoting the spectral norm by $\|\cdot\|_2$ it follows
that there are $\varepsilon
> 0$ and $K>0$ such that
$$\|M'(z)^*\|_2 \geq K |z-z_0|^{-1} \quad \forall\; z\in \bC:
|z|=1, z\neq z_0, |z-z_0| \leq \varepsilon;$$ this may be seen by
considering first the row sum norm of $M'(z)^*$ and then using the
equivalence of norms. Since the matrix $M'(z) M'(z)^*$ is hermitian,
we conclude that
$$\| M'(z) M'(z)^*\|_2 = \sup_{v\in \bC^n: |v|=1} |v^*M'(z) M'(z)^* v| =
\sup_{v\in \bC^n: |v|=1} |M'(z)^* v|^2 \geq K^2 |z-z_0|^2$$ for all
$z\neq z_0$ on the unit circle such that $|z-z_0| \leq \varepsilon$.
But this implies that $\omega \mapsto M'(e^{-i\omega}) M'(e^{-i
\omega})^*$ cannot be integrable on $(-\pi,\pi] \setminus N$, giving
the desired contradiction. This finishes the proof of
Theorem~\ref{thm-4}.

\section{Proof  of Theorem~\ref{thm-5}} \label{S6}
\setcounter{equation}{0}

In this section we shall prove  Theorem~\ref{thm-5}.  For that, we
first observe that ARMA$(p,q)$ equations can be embedded into higher
dimensional ARMA$(1,q)$ processes, as stated in the following
proposition. This is well known and its proof is immediate, hence
omitted.

\begin{proposition} \label{thm-2}
Let $m,d, p\in \bN$, $q\in \bN_0$, and let $(Z_t)_{t\in \bZ}$ be an
i.i.d. sequence of $\bC^d$-valued random vectors. Let $\Psi_1,
\ldots, \Psi_p \in \bC^{m\times m}$ and $\Theta_0, \ldots, \Theta_q
\in \bC^{m\times d}$ be complex-valued matrices. Define the matrices
$\underline{\Phi} \in \bC^{mp\times mp}$ and $\underline{\Theta}_k
\in \bC^{mp \times d}$, $k\in \{0,\ldots, q\}$, by
\begin{equation}
\underline{\Phi} := \begin{pmatrix} \Psi_1 & \Psi_2 & \cdots &
\Psi_{p-1} & \Psi_p \\ \mbox{\rm Id}_m & 0_{m,m} & \cdots & 0_{m,m}
& 0_{m,m} \\
0_{m,m} &  \ddots & \ddots & \vdots & \vdots\\
\vdots & \ddots & \ddots & 0_{m,m} & \vdots \\
0_{m,m} & \cdots & 0_{m,m} & \mbox{\rm Id}_m & 0_{m,m}
\end{pmatrix} \quad \mbox{and}\quad
\underline{\Theta}_k = \begin{pmatrix} \Theta_k  \\
0_{m,d}  \\
\vdots  \\
0_{m,d}  \end{pmatrix}. \label{eq-def-Phi}
\end{equation}
 Then
 the
ARMA$(p,q)$ equation \eqref{eqpq} admits a strictly stationary
solution $(Y_t)_{t\in \bZ}$ of $m$-dimensional random vectors $Y_t$
if and only if the ARMA$(1,q)$ equation
\begin{equation} \label{eq1q-gross}
\underline{Y}_t - \underline{\Phi} \, \underline{Y}_{t-1} =
\underline{\Theta}_0 {Z}_t + \underline{\Theta}_1 {Z}_{t-1} + \ldots
+ \underline{\Theta}_q {Z}_{t-q}, \quad t\in \bZ,
\end{equation}
admits a strictly stationary solution $(\underline{Y}_t)_{t\in \bZ}$
of $mp$-dimensional random vectors $\underline{Y}_t$. More
precisely, if $(Y_t)_{t\in \bZ}$ is a strictly stationary solution
of \eqref{eqpq}, then
\begin{equation} \label{eq-gross-Y}
(\underline{Y}_t)_{t\in \bZ} := ((Y_t^T , Y_{t-1}^T , \ldots,
Y_{t-(p-1)}^T)^T)_{t\in \bZ}
\end{equation}
is a strictly stationary solution of \eqref{eq1q-gross}, and
conversely, if $(\underline{Y}_t)_{t\in \bZ} = (({Y_t^{(1)T}},
\ldots, {Y_{t}^{(p) T}})^T)_{t\in \bZ}$ with random components
$Y_t^{(i)} \in \bC^m$ is a strictly stationary solution of
\eqref{eq1q-gross}, then $(Y_t)_{t\in \bZ} := (Y_{t}^{(1)})_{t\in
\bZ}$ is a strictly stationary solution of \eqref{eqpq}.
\end{proposition}

For the proof of Theorem~\ref{thm-5} we need some notation: define
$\underline{\Phi}$ and $\underline{\Theta}_k$ as in
\eqref{eq-def-Phi}. Choose an invertible $\bC^{mp\times mp}$ matrix
$\underline{S}$ such that $\underline{S}^{-1} \underline{\Phi}
\underline{S}$ is in Jordan canonical form, with ${H}$ Jordan blocks
$\underline{\Phi}_1, \ldots, \underline{\Phi}_{H}$, say, the
$h^{th}$ Jordan block $\underline{\Phi}_h$ starting in row
$\underline{r}_{h}$, with $\underline{r}_1 := 1 < \underline{r}_2 <
\cdots < \underline{r}_H < mp+1 =: \underline{r}_{H+1}$. Let
$\underline{\lambda}_h$ be the eigenvalue associated with
$\underline{\Phi}_h$, and, similarly to \eqref{def-Is}, denote by
$\underline{I}_h$ the $(\underline{r}_{h+1} - \underline{r}_h)\times
mp$-matrix with components $\underline{I}_h (i,j) = 1$ if $j=i +
\underline{r}_h -1$ and $\underline{I}_h (i,j) = 0$ otherwise. For
$h\in \{1,\ldots, H\}$ and $j\in \bZ$ let
\begin{equation*}
N_{j,h} :=  \begin{cases} \mathbf{1}_{j\geq 0}
 \underline{\Phi}_{h}^{j-q}
\sum_{k=0}^{j\wedge q} \underline{\Phi}_{h}^{q-k} \underline{I}_{h}
\underline{S}^{-1} \underline{\Theta}_k, & |\underline{\lambda}_h|
\in (0,1),
\\
 -\mathbf{1}_{j\leq q-1}  \underline{\Phi}_{h}^{j-q}
\sum_{k=(1+j)\vee 0}^{q} \underline{\Phi}_{h}^{q-k}
\underline{I}_{h} \underline{S}^{-1}
\underline{\Theta}_k , & |\underline{\lambda}_h| > 1,\\
\mathbf{1}_{j\in \{0,\ldots, mp+q-1\}}  \sum_{k=0}^{j\wedge q}
\underline{\Phi}_h^{j-k} \underline{I}_h
\underline{S}^{-1}\underline{\Theta}_k, & \underline{\lambda}_h = 0,\\
\mathbf{1}_{j\in \{0,\ldots, q-1\}}
\sum_{k=0}^j\underline{\Phi}_{h}^{j-k}\underline{I}_{h}
\underline{S}^{-1}\underline{\Theta}_k , & |\underline{\lambda}_h| =
1,
\end{cases} \end{equation*} and
\begin{equation} \label{def-Nj}
\underline{N}_j := \underline{S}^{-1} (N_{j,1}^T, \ldots,
N_{j,H}^T)^T \in \bC^{mp \times d}.
\end{equation}
Further, let
 $U$ and $K$ be defined as in the statement of the theorem, and
denote  $$W_t := U Z_t, \quad t\in\bZ.$$  Then $(W_t)_{t\in \bZ}$ is
an i.i.d. sequence. Equation~\eqref{eq-uw} is then an easy
consequence of the fact that for $a\in \bC^d$ the distribution of
$a^*W_0 = (U^* a)^* Z_0$ is degenerate to a Dirac measure if and
only if $U^* a \in K$, i.e. if $a \in UK = \{0_s\} \times
\bC^{d-s}$: taking for $a$ the $i^{th}$ unit vector in $\bC^d$ for
$i\in \{s+1,\ldots, d\}$, we see that $W_t$ must be of the form
$(w_t^T,u^T)^T$ for some $u\in \bC^{d-s}$, and taking $a= (b^T,
0_{d-s}^T)^T$ for $b\in \bC^{s}$ we see that $b^* w_0$ is not
degenerate to a Dirac measure for $b\neq 0_{s}$. The remaining proof
of the necessity of the conditions, the sufficiency of the
conditions and the stated uniqueness will be given in the next
subsections.

\subsection{The necessity of the conditions} \label{S-5-1}

Suppose that $(Y_t)_{t\in \bZ}$ is a strictly stationary solution of
\eqref{eqpq}. Define $\underline{Y}_t$ by \eqref{eq-gross-Y}.  Then
$(\underline{Y}_t)_{t\in \bZ}$ is a strictly stationary solution of
\eqref{eq1q-gross} by Proposition~\ref{thm-2}. Hence, by
Theorem~\ref{thm-main}, there is $\underline{f}' \in \bC^{mp}$, such
that $(\underline{Y}_t')_{t\in \bZ}$, defined by
\begin{equation} \label{eq-Yt-prime}
\underline{Y}_t' = \underline{f}' + \sum_{j=-\infty}^\infty
\underline{N}_j Z_{t-j}, \quad t \in \bZ,
\end{equation}
is (possibly another) strictly stationary solution of
$$\underline{Y}_t' - \underline{\Phi}\, \underline{Y}_{t-1}' =
\sum_{k=0}^q \underline{\Theta}_k Z_{t-k} = \sum_{k=0}^q
\widetilde{\underline{\Theta}}_k W_{t-k}, \quad t \in \bZ,$$ where
$\widetilde{\underline{\Theta}}_k := \underline{\Theta}_k U^*$. The
sum in \eqref{eq-Yt-prime} converges almost surely absolutely. Now
define $A_h \in \bC^{(\underline{r}_{h+1} - \underline{r}_h) \times
s}$ and $C_h \in \bC^{(\underline{r}_{h+1} - \underline{r}_h) \times
(d-s)}$  for $h\in \{1,\ldots, H´\}$ such that
$|\underline{\lambda}_h|=1$ by
\begin{equation} \label{eq-AB}
(A_h , C_h) := \sum_{k=0}^q \underline{\Phi}_h^{q-k} \underline{I}_h
\underline{S}^{-1} \widetilde{\underline{\Theta}}_k .
\end{equation}
By conditions (ii) and (iii) of Theorem~\ref{thm-main}, for every
such $h$ with $|\underline{\lambda}_h|=1$ there exists a vector
$\underline{\alpha}_h = (\alpha_{h,1}, \ldots,
\alpha_{h,\underline{r}_{h+1}-\underline{r}_h})^T \in
\bC^{\underline{r}_{h+1} - \underline{r}_h}$ such that
$$(A_h, C_h) W_0 = \underline{\alpha}_h \quad
\mbox{a.s.}$$ with $\alpha_{h,1} = 0$ if $\underline{\lambda}_h=1$.
Since $W_0 = (w_0^T , u^T)^T$, this implies $A_h w_0 =
\underline{\alpha}_h - C_h u$, but since $b^* w_0$ is not degenerate
to a Dirac measure for any $b\in \bC^s \setminus \{ 0_s\}$, this
gives $A_h=0$ and hence $C_h u = \underline{\alpha}_h$ for $h\in
\{1,\ldots, H\}$ such that $|\underline{\lambda}_h|=1$. Now let
$v\in \bC^s$ and $(W_t'')_{t\in \bZ}$ be an i.i.d. $N( \left(
\begin{array}{c} v \\ u \end{array} \right) , \left( \begin{array}{ll}
\mbox{\rm Id}_s & 0_{s,d-s} \\ 0_{d-s,s} & 0_{d-s,d-s} \end{array}
\right))$-distributed sequence, and let $Z_t'' := U^* W_t''$.
 Then
$$ (A_h, C_h) W_0'' = C_h u = \underline{\alpha}_h \quad \mbox{a.s.} \quad \mbox{for}\quad h\in \{1,\ldots, H\}:
|\underline{\lambda}_h|=1$$ and
$$\bE \log^+ \left\| \sum_{k=0}^q
\underline{\Phi}_h^{q-k} \underline{I}_h \underline{S}^{-1}
\widetilde{\underline{\Theta}}_k W_0''\right\|< \infty \quad
\mbox{for}\quad h\in \{1,\ldots, H\}: |\underline{\lambda}_h|\neq
0,1.$$
 It then follows from
Theorem~\ref{thm-main} that there is a strictly stationary solution
$\underline{Y}_t''$ of the ARMA$(1,q)$ equation $\underline{Y}_t'' -
\underline{\Phi} \, \underline{Y}_{t-1}'' = \sum_{k=0}^q
\widetilde{\underline{\Theta}}_k W_{t-k}''=\sum_{k=0}^q
\underline{\Theta}_k Z_{t-k}''$, which can be written in the form
$\underline{Y}_t'' = \underline{f}''+ \sum_{j=-\infty}^\infty
\underline{N}_j Z_{t-j}''$ for some $\underline{f}'' \in \bC^{mp}$.
In particular, $(\underline{Y}_t'')_{t\in \bZ}$ is a Gaussian process.
Again from Proposition~\ref{thm-2} it follows that there is a
Gaussian process $(Y_t'')_{t\in \bZ}$ which is a strictly stationary
solution of
$$ Y_t'' - \sum_{k=1}^p \Psi_k Y_{t-k}'' = \sum_{k=0}^q
{\widetilde{\Theta}}_k W_{t-k}'' = \sum_{k=0}^q {\Theta}_k Z_{t-k}''
, \quad t \in \bZ.$$  In particular, this solution is also weakly
stationary. Hence it follows from Theorem~\ref{thm-4} that
$z\mapsto M(z)$ has only removable singularities on the unit circle
and that \eqref{eq-g} has a solution $g\in \bC^{m}$, since $\bE
Z_0'' = U^* (v^T, u^T)^T$. Hence we have established that (i) and
(iii'), and hence (iii), of Theorem~\ref{thm-5} are necessary
conditions for a strictly stationary solution to exist.

 To see the necessity of
conditions (ii) and (ii'), we need the following lemma, which is
interesting in itself since it expresses the Laurent coefficients of
$M(z)$ in terms of the Jordan canonical decomposition of
$\underline{\Phi}$.

\begin{lemma} \label{lem-3}
With the notations of Theorem~\ref{thm-5} and those introduced after
Proposition~\ref{thm-2}, suppose that condition (i) of
Theorem~\ref{thm-5} holds, i.e. that $M(z)$ has only removable
singularities on the unit circle. Denote by $M(z) =
\sum_{j=-\infty}^\infty M_j z^j$ the Laurent expansion of $M(z)$ in
a neighborhood of the unit circle. Then
\begin{equation} \label{def-Mj}
\underline{M}_j := ( M_j^T ,M_{j-1}^T , \ldots, M_{j-p+1}^T)^T =
\underline{N}_j U^* \left( \begin{array}{ll} \mbox{\rm Id}_s &
0_{s,d-s} \\ 0_{d-s,s} & 0_{d-s,d-s} \end{array} \right) \quad
\forall\; j \in \bZ.
\end{equation}
In particular,
\begin{equation} \label{eq-Nt2}
\underline{M}_j U Z_{t-j} = \underline{N}_j  Z_{t-j} -
\underline{N}_j U^* (0_s^T, u^T)^T \quad \forall\; j,t \in \bZ.
\end{equation}
\end{lemma}

\begin{proof} Define $\Lambda := \left( \begin{array}{ll} \mbox{\rm Id}_s & 0_{s,d-s} \\ 0_{d-s,s} &
0_{d-s,d-s} \end{array} \right)$ and  let $(Z_t')_{t\in \bZ}$ be an
i.i.d. $N(0_d, U^* \Lambda U)$-distri-buted noise sequence and define
$Y_t' := \sum_{j=-\infty}^\infty M_j U Z_{t-j}'$. Then $(Y_t')_{t\in
\bZ}$ is a weakly and strictly stationary solution of $P(B) Y_t' =
Q(B) Z_t'$ by Theorem~\ref{thm-4}, and the entries of $M_j$ decrease
geometrically as $|j|\to \infty$. By Proposition~\ref{thm-2}, the
process $(\underline{Y}_t')_{t\in \bZ}$ defined by $\underline{Y}_t'
= ({Y_t'}^T, {Y_{t-1}'}^T, \ldots, {Y_{t-p+1}'}^T) =
\sum_{j=-\infty}^\infty \underline{M}_j U Z_{t-j}'$ is a strictly
stationary solution of
\begin{equation} \label{eq-gross-Q} \underline{Y}_t' -
\underline{\Phi} \, \underline{Y}_{t-1}' = \sum_{j=0}^q
\underline{\Theta}_j Z_{t-j}', \quad t \in \bZ.
\end{equation}
Denoting $\underline{\Theta}_j = 0_{mp,d}$ for $j\in \bZ \setminus
\{0,\ldots, q\}$, it follows that $\sum_{k=-\infty}^\infty
(\underline{M}_k - \underline{\Phi} \,\underline{M}_{k-1}) U
Z_{t-k}' = \sum_{k=-\infty}^\infty \underline{\Theta}_k Z_{t-k}'$,
and multiplying this equation from the right by ${Z'}_{t-j}^T$,
taking expectations and observing that $M(z)\Lambda=M(z)$ we conclude that
\begin{equation} \label{eq-L1}(\underline{M}_j - \underline{\Phi}\,
\underline{M}_{j-1}) U  = (\underline{M}_j - \underline{\Phi}\,
\underline{M}_{j-1}) \Lambda U= \underline{\Theta}_j U^* \Lambda U
\quad \forall\;j\in \bZ. \end{equation}

Next observe that since $(\underline{Y}_t')_{t\in \bZ}$ is a
strictly stationary solution of \eqref{eq-gross-Q}, it follows from
Theorem~\ref{thm-main} that $(\underline{Y}_t'')_{t\in \bZ}$,
defined by $\underline{Y}_t'' = \sum_{j=-\infty}^\infty
\underline{N}_j Z_{t-j}'$, is also a strictly stationary solution of
\eqref{eq-gross-Q}. With precisely the same argument as above it
follows that \begin{equation} \label{eq-L2} (\underline{N}_j -
\underline{\Phi} \,\underline{N}_{j-1}) U^* \Lambda U  =
\underline{\Theta}_j U^* \Lambda U \quad \forall\; j\in \bZ.
\end{equation} Now let $L_j := \underline{M}_j - \underline{N}_j U^*
\Lambda$, $j\in \bZ$. Then $L_j - \underline{\Phi} L_{j-1} =
0_{mp,d}$ from \eqref{eq-L1} and \eqref{eq-L2}, and the entries of
$L_j$ decrease exponentially as $|j|\to \infty$ since so do the
entries of $\underline{M}_j$ and $\underline{N}_j$. It follows that
for $h\in \{1,\ldots, H\}$ and $j\in \bZ$ we have
\begin{equation} \label{eq-Q-gross-2} \underline{I}_{h}
\underline{S}^{-1} L_j - \underline{\Phi}_h \underline{I}_{h}
\underline{S}^{-1} L_{j-1} = \underline{I}_{h} \left(
\underline{S}^{-1} L_j - \begin{pmatrix} \underline{\Phi}_1 &  &\\ & \ddots & \\
& & \underline{\Phi}_H \end{pmatrix} \underline{S}^{-1} L_{j-1}
\right)= 0_{\underline{r}_{h+1}-\underline{r}_h,d}.\end{equation}
Since $\underline{\Phi}_h$ is invertible for $h\in \{1,\ldots, H\}$
such that $\underline{\lambda}_h\neq 0$, this gives
$\underline{I}_{h} \underline{S}^{-1} L_0 =
\underline{\Phi}_{h}^{-j} \underline{I}_{h} \underline{S}^{-1} L_j$
for all $j\in \bZ$ and $\underline{\lambda}_h\neq 0$.
 Since for $|\underline{\lambda}_h|\geq 1$, $\|\underline{\Phi}_{h}^{-j}\| \leq \kappa j^{mp}$ for all $j\in
\bN_0$ for some constant $\kappa$, it follows that
$\|\underline{I}_{h} \underline{S}^{-1} L_0 \| \leq \kappa j^{mp}
\|\underline{I}_h \underline{S}^{-1} L_j\|$, which converges to 0 as
$j\to\infty$ by the geometric decrease of the coefficients of $L_j$
as $j\to\infty$, so that $\underline{I}_{h} \underline{S}^{-1} L_k =
0$ for $|\underline{\lambda}_h|\geq 1$ and $k=0$ and hence for all
$k\in \bZ$. Similarly, letting $j\to-\infty$, it follows that
$\underline{I}_h \underline{S}^{-1} L_k = 0$ for
$|\underline{\lambda}_h|\in (0,1)$ and $k=0$ and hence for all $k\in
\bZ$. Finally, for $h\in \{1,\ldots, H\}$ such that
$\underline{\lambda}_h=0$ observe that $\underline{I}_{h}
\underline{S}^{-1} L_k = \underline{\Phi}_{h}^{mp} \underline{I}_{h}
\underline{S}^{-1} L_{k-mp}$ for $k\in \bZ$ by \eqref{eq-Q-gross-2},
and since $\underline{\Phi}_{h}^{mp} = 0$, this shows that
$\underline{I}_{h} \underline{S}^{-1} L_k = 0$ for $k\in \bZ$.
Summing up, we have $\underline{S}^{-1} L_k =0$ and hence
$\underline{M}_k = \underline{N}_k U^* \Lambda$ for $k\in \bZ$,
which is \eqref{def-Mj}. Equation~\eqref{eq-Nt2} then follows from
\eqref{eq-uw}, since
$$\underline{M}_j U Z_{t-j} = \underline{M}_j \left(
\begin{array}{c}
w_{t-j} \\ u \end{array} \right) = \underline{N}_j U^* \left(
\begin{array}{c} w_{t-j} \\ 0_{d-s} \end{array} \right) =
\underline{N}_j U^* \left(U Z_{t-j} -  \left( \begin{array}{c} 0 \\
u
\end{array} \right) \right).$$
\end{proof}

Returning to the proof of the necessity of conditions (ii) and (ii')
for a strictly stationary solution to exist, observe that
 $\sum_{j=-\infty}^\infty
\underline{N}_j Z_{t-j}$ converges almost surely absolutely by
\eqref{eq-Yt-prime}, and since the entries of $\underline{N}_j$
decrease geometrically as $|j|\to\infty$, this together with
\eqref{eq-Nt2} implies that $\sum_{j=-\infty}^\infty \underline{M}_j
U Z_{t-j}$ converges almost surely absolutely, which shows that
(ii') must hold. To see (ii), observe that for $j\geq mp +q$ we have
\begin{equation*} \label{N3} N_{j,h} =
\begin{cases} \underline{\Phi}_{h}^{j-q} \sum_{k=0}^{q}
\underline{\Phi}_{h}^{q-k} \underline{I}_{h} \underline{S}^{-1}
\underline{\Theta}_k, & |\underline{\lambda}_h| \in (0,1),\\
0, & |\underline{\lambda}_h| \not\in (0,1),
\end{cases}\end{equation*} while
\begin{equation*} \label{eq-N4}
N_{-1,h} = \begin{cases} \underline{\Phi}_{h}^{-1-q} \sum_{k= 0}^{q}
\underline{\Phi}_{h}^{q-k} \underline{I}_{h} \underline{S}^{-1}
\underline{\Theta}_k , & |\underline{\lambda}_h| > 1,\\
0, & |\underline{\lambda}_h| \leq 1.\end{cases} \end{equation*}
Since a strictly stationary solution of \eqref{eq1q-gross} exists,
it follows from Theorem~\ref{thm-main} that $\bE \log^+
\|\underline{N}_{j} Z_0\| < \infty$ for $j\geq mp+q$ and $\bE \log^+
\| \underline{N}_{-1} Z_0\| < \infty$. Together with \eqref{eq-Nt2}
this shows that condition (ii) of Theorem~\ref{thm-5} is necessary.

\subsection{The sufficiency of the conditions and uniqueness of the
solution}

In this subsection we shall show that (i), (ii), (iii) as well as
(i), (ii'), (iii) of Theorem~\ref{thm-5} are sufficient conditions
for a strictly stationary solution of \eqref{eqpq} to exist, and
prove the uniqueness assertion.

(a) Assume that conditions (i), (ii) and (iii) hold for some $v\in
\bC^s$ and $g\in \bC^m$. Then $\bE \log^+ \| \underline{N}_{-1}
Z_0\| < \infty$ and $\bE \log^+ \| \underline{N}_{mp+q} Z_0\| <
\infty$ by (ii) and \eqref{eq-Nt2}. In particular, since
$\underline{S}$ is invertible, $\bE \log^+ \| N_{-1,h} Z_0\| <
\infty$ for $|\underline{\lambda}_h|> 1$ and $\bE \log^+ \|
N_{mp+q,h} Z_0\| < \infty$ for $|\underline{\lambda}_h| \in (0,1)$.
The invertibility of $\underline{\Phi}_h$ for $\underline{\lambda}_h
\neq 0$ then shows that
\begin{equation} \label{eq-cond-ii-new}
\bE \log^+ \left\| \sum_{k=0}^q \underline{\Phi}_h^{q-k}
\underline{I}_h \underline{S}^{-1} \underline{\Theta}_k Z_0 \right\|
< \infty \quad \forall\; h\in \{1,\ldots, H\}:
|\underline{\lambda}_h| \in (0,1) \cup (1,\infty).
\end{equation}
Now let $(W_t''')_{t\in \bZ}$ be an i.i.d. $N( \left(
\begin{array}{c} v \\ u \end{array} \right) , \left(
\begin{array}{ll} \mbox{\rm Id}_s & 0_{s,d-s} \\ 0_{d-s,s} &
0_{d-s,d-s} \end{array} \right))$ distributed sequence and define
$Z_t''' := U^* W_t'''$. Then $\bE Z_t''' = U^* (v^T, u^T)^T$. By
conditions (i) and (iii) and  Theorem~\ref{thm-4}, $(Y_t''')_{t\in
\bZ}$, defined by $Y_t''' := P(1)^{-1} Q(1) \bE Z_0''' +$
$\sum_{j=-\infty}^\infty M_j (W_{t-j}'''- (v^T,u^T)^T)$,  is a
weakly stationary solution of $Y_t''' - \sum_{k=1}^p \Psi_k
Y_{t-k}''' = \sum_{k=0}^q \Theta_k Z_{t-k}'''$, and obviously, it is
also strictly stationary. It now follows in complete analogy to the
necessity proof presented in Section~\ref{S-5-1} that $A_h=0$ and
$C_h u = (\alpha_{h,1}, \ldots, \alpha_{h,\underline{r}_{h+1} -
\underline{r}_h})^T$ for $|\underline{\lambda}_h|=1$, where
$(A_h,C_h)$ is defined as in \eqref{eq-AB} and $\alpha_{h,1} = 0$ if
$\lambda_h=1$. Hence $\sum_{k=0}^q \underline{\Phi}_h^{q-k}
\underline{I}_h \underline{S}^{-1} \underline{\widetilde{\Theta}}_k
W_0 = (\alpha_{h,1}, \ldots, \alpha_{h,\underline{r}_{h+1} -
\underline{r}_h})^T$ for $|\underline{\lambda}_h|=1$. By
Theorem~\ref{thm-main}, this together with \eqref{eq-cond-ii-new}
implies the existence of a strictly stationary solution of
\eqref{eq1q-gross},  so that a strictly stationary solution
$(Y_t)_{t\in \bZ}$ of \eqref{eqpq} exists by
Proposition~\ref{thm-2}.

(b) Now assume that conditions (i), (ii') and (iii) hold for some
$v\in \bC^s$ and $g\in \bC^m$ and define $Y=(Y_t)_{t\in \bZ}$ by
\eqref{eq-Y}. Then $Y$ is clearly strictly stationary. Since $U Z_t
= (w_t^T, u^T)$, we further have, using (iii), that
\begin{eqnarray*}
P(B) Y_t  & = &  P(1) g - P(1)M(1) \left( \begin{array}{c} v \\ u
\end{array} \right) + Q(B) U^* \left(
\begin{array}{ll} \mbox{\rm Id}_s & 0_{s,d-s} \\ 0_{d-s,s} & 0_{d-s,d-s}
\end{array} \right) \left( \begin{array}{c} w_t \\ u \end{array}
\right) \\
& = &  Q(1) U^* \left(
\begin{array}{c} v \\ u \end{array} \right)
 - Q(1) U^* \left(
\begin{array}{c} v \\ 0_{d-s} \end{array} \right)
+ Q(B) U^* \left(
\begin{array}{c} w_t \\ 0_{d-s} \end{array} \right)
\\
& = &  Q(B) U^* \left( \begin{array}{c} w_t \\ u \end{array} \right)
= Q(B) Z_t
\end{eqnarray*}
for $t\in \bZ$, so that $(Y_t)_{t\in \bZ}$ is a solution of
\eqref{eqpq}.

(c) Finally,  the uniqueness assertion follows from the fact that by
Proposition~\ref{thm-2}, \eqref{eqpq} has a unique strictly
stationary solution if and only if \eqref{eq1q-gross} has a unique
strictly stationary solution. By Theorem~\ref{thm-main}, the latter
is equivalent to the fact that $\underline{\Phi}$ does not have an
eigenvalue on the unit circle, which in turn is equivalent to $\det
P(z) \neq 0$ for $z$ on the unit circle, since  $\det P(z) = \det
(\mbox{\rm Id}_{mp} - \underline{\Phi} z)$ (e.g. Gohberg et
al.~\cite{GLR}, p.~14). This finishes the proof of
Theorem~\ref{thm-5}.

\section{Discussion and consequences of main results} \label{S7}
\setcounter{equation}{0}

In this section we shall discuss the main results and consider
special cases. Some consequences of the results are also listed. We
start with some comments  on Theorem~\ref{thm-main}. If $\Psi_1$ has
only eigenvalues of absolute value in $(0,1)\cup (1,\infty)$, then a
much simpler condition for stationarity of \eqref{eq1q} can be
given:

\begin{corollary} \label{cor-1}

Let the assumptions of Theorem~\ref{thm-main} be satisfied and
suppose that
 $\Psi_1$ has only eigenvalues of absolute value in $(0,1) \cup
(1,\infty)$. Then a strictly stationary solution of \eqref{eq1q}
exists if and only if
\begin{equation} \label{bed3}
\mathbb{E} \log^+ \left\| \left( \sum_{k=0}^q \Psi_1^{q-k}
\Theta_k\right)Z_0 \right\|  <  \infty.
 \end{equation}
\end{corollary}

\begin{proof}
It follows from Theorem~\ref{thm-main} that there exists a strictly
stationary solution if and only if \eqref{bed1a} holds for every
$h\in \{1,\ldots, H\}$. But this is equivalent to $$\mathbb{E}
\log^+ \| ( \sum_{k=0}^q(S^{-1} \Psi_1 S)^{q-k} \mbox{Id}_m S^{-1}
\Theta_k)Z_0 \| < \infty,$$ which in turn is equivalent to
\eqref{bed3}, since $S$ is invertible and hence for a random vector
$R \in \bC^m$ we have $\mathbb{E} \log^+ \|S R\|<\infty$ if and only
if $\mathbb{E} \log^+ \| R\| < \infty$.
\end{proof}

\begin{remark} \label{rem-2}
Suppose that $\Psi_1$ has only eigenvalues of absolute value in
$(0,1) \cup (1,\infty)$. Then $\mathbb{E} \log^+ \|Z_0\|$ is a
sufficient condition for \eqref{eq1q} to have a strictly stationary
solution, since it implies \eqref{bed3}. But it is not necessary.
For example, let $q=1$, $m=d=2$ and
$$\Psi_1 = \begin{pmatrix} 2 & 0
\\ 0 & 3 \end{pmatrix}, \quad \Theta_0 = \mbox{\rm Id}_2, \quad \Theta_1
= \begin{pmatrix} -1 & -1 \\ 1 & -4 \end{pmatrix}, \quad \mbox{so
that} \quad \sum_{k=0}^1 \Psi_1^{q-k} \Theta_k =
\begin{pmatrix} 1 & -1
\\ 1 & -1
\end{pmatrix}.$$
By \eqref{bed3},  a strictly stationary solution exists for example
if the i.i.d. noise $(Z_t)_{t\in \bZ}$ satisfies $Z_0 = (R_0, R_0 +
R_0')^T$, where $R_0'$ is a random variable with finite log moment
and $R_0$ a random variable with infinite log moment. In particular,
 $\mathbb{E} \log^+ \|Z_0\| = \infty$ is possible.
\end{remark}

An example like in the remark above cannot occur if the matrix
$\sum_{k=0}^q \Psi_1^{q-k} \Theta_k$ is invertible if $m=d$. More
generally, we have the following result:

\begin{corollary} \label{cor-2}
Let the assumptions of Theorem~\ref{thm-main} be satisfied and
suppose that
 $\Psi_1$ has only eigenvalues of absolute value in $(0,1) \cup
(1,\infty)$. Suppose further that $d\leq m$ and that $\sum_{k=0}^q
\Psi_1^{q-k} \Theta_k$ has full rank $d$. Then a strictly stationary
solution of \eqref{eq1q} exists if and only if $\mathbb{E} \log^+
\|Z_0\| < \infty$.
\end{corollary}

\begin{proof}
The sufficiency of the condition has been observed in
Remark~\ref{rem-2}, and for the necessity, observe that with $A:=
\sum_{k=0}^q \Psi_1^{q-k} \Theta_k$ and $U:= A Z_0$ we must have
$\mathbb{E} \log^+ \|U\| < \infty$ by \eqref{bed3}. Since $A$ has
rank $d$, the matrix $A^T A \in \bC^{d\times d}$ is invertible and
we have $Z_0 = (A^T A)^{-1} A^T U$, i.e. the components of $Z_0$ are
linear combinations of those of $U$. It follows that $\mathbb{E}
\log^+ \|Z_0\| < \infty$.
\end{proof}

Next, we shall discuss the conditions of Theorem~\ref{thm-5} in more
detail. The following remark is obvious from Theorem~\ref{thm-5}. It
implies in particular the well known fact that $\bE \log^+
\|Z_0\|<\infty$ together with $\det P(z)\neq 0$ for all $z$ on the
unit circle is sufficient for the existence of a strictly stationary
solution.

\begin{remark} \label{rem-2a} (a) $\bE \log^+ \|Z_0\|<\infty$ is a
sufficient condition for (ii) of Theorem~\ref{thm-5}.\\
(b) $\det P(1) \neq 0$ is a sufficient condition for (iii) of
Theorem~\ref{thm-5}.\\
(c) $\det P(z) \neq 0$ for all $z$ on the unit circle is a
sufficient condition for (i) and (iii) of Theorem~\ref{thm-5}.
\end{remark}

With the notations of Theorem~\ref{thm-5}, denote
\begin{equation} \label{def-tilde-Q}
\widetilde{Q}(z) := Q(z) U^* \left(
\begin{array}{ll} \mbox{\rm Id}_s & 0_{s,d-s}
\\ 0_{d-s,s} & 0_{d-s,d-s} \end{array} \right) ,
\end{equation}
so that $M(z) = P^{-1}(z) \widetilde{Q}(z)$. It is natural to ask if
conditions (i) and (iii) of Theorem~\ref{thm-5} can be replaced by a
removability condition on the singularities on the unit circle of
$(\det P(z))^{-1} \det (\widetilde{Q}(z))$ if $d=m$. The following
corollary shows that this condition is indeed necessary, but it is
not sufficient as pointed out in Remark~\ref{rem-2c}.

\begin{corollary} \label{cor-3}
Under the assumptions of Theorem~\ref{thm-main}, with
$\widetilde{Q}(z)$ as defined in \eqref{def-tilde-Q},  a necessary
condition for a strictly stationary solution of the ARMA$(p,q)$
equation \eqref{eqpq} to exist is that the function $z \mapsto |\det
{P}(z)|^{-2} {\det (\widetilde{Q}(z)\widetilde{Q}(z)^*)}$ has only
removable singularities on the unit circle. If additionally $d=m$,
then a necessary condition for a strictly stationary solution to
exist is that the matrix rational  function $z \mapsto (\det
P(z))^{-1} \det (\widetilde{Q}(z))$ has only removable singularities
on the unit circle.
\end{corollary}

\begin{proof} The second assertion  is immediate from
Theorem~\ref{thm-5}, and the first assertion follows from the fact
that if $M(z)$ as defined in Theorem~\ref{thm-5} has only removable
singularities on the unit circle, then so does $M(z) M(z)^*$ and
hence $\det (M(z) M(z)^*)$.
\end{proof}

\begin{remark} \label{rem-2c}
In the case $d=m$ and $\bE \log^+\|Z_0\| < \infty$, the condition
that the matrix rational function $z \mapsto (\det {P}(z))^{-1}
{\det \widetilde{Q}(z)}$ has only removable singularities on the
unit circle is not sufficient for the existence of a strictly
stationary solution of \eqref{eq1q}. For example, let $p=q=1$,
$m=d=2$ and $\Psi_1 = \Theta_0 = \mbox{\rm Id}_{2}$, $\Theta_1 =
\begin{pmatrix} -1 &  0\\ 1 & -1 \end{pmatrix}$, $(Z_t)_{t\in \bZ}$ be
i.i.d. standard normally distributed and $U=\mbox{\rm Id}_2$. Then
$\det P(z) = \det \widetilde{Q}(z) = (1-z)^2$, but it does not hold
that $\Psi_1 \Theta_0 + \Theta_1 = 0$, so that condition (iii) of
Theorem~\ref{thm-main} is violated and no strictly stationary
solution can exist.
\end{remark}

Next, we shall discuss condition (i) of Theorem~\ref{thm-5} in more
detail.  Recall (e.g. Kailath \cite{Kailath}) that a  $\bC^{m\times
m}$ matrix polynomial $R(z)$ is a {\it left-divisor} of $P(z)$, if
there is a matrix polynomial $P_1(z)$ such that $P(z) = R(z)
P_1(z)$. The matrix polynomials $P(z)$ and $\widetilde{Q}(z)$ are
{\it left-coprime}, if every common left-divisor $R(z)$ of $P(z)$
and $\widetilde{Q}(z)$ is {\it unimodular}, i.e. the determinant of
$R(z)$ is constant in $z$. In that case, the matrix rational
function $P^{-1}(z) \widetilde{Q}(z)$ is also called {\it
irreducible}. With $\widetilde{Q}$ as defined in
\eqref{def-tilde-Q}, it is then easy to see that condition (i) of
Theorem~\ref{thm-5} is equivalent to {\it
\begin{enumerate}
\item[(i')] There exist $\bC^{m\times m}$-valued matrix polynomials
$P_1(z)$ and $R(z)$ and a $\bC^{m\times d}$-valued matrix polynomial
$Q_1(z)$ such that $P(z) = R(z) P_1(z)$, $\widetilde{Q}(z) = R(z)
Q_1(z)$ for all $z\in \bC$ and $\det P_1(z) \neq 0$ for all $z$ on
the unit circle.
\end{enumerate} }
That (i') implies (i) is obvious, and that (i) implies (i')
follows by taking $R(z)$ as the greatest common left-divisor
(cf.~\cite{Kailath}, p.~377) of $P(z)$ and $\widetilde{Q}(z)$. The
thus remaining right-factors $P_1(z)$ and $Q_1(z)$ are then left-coprime,
and since the matrix rational function $M(z) = P^{-1}(z)
\widetilde{Q}(z) = P_1^{-1}(z) Q_1(z)$ has no poles on the unit
circle, it follows from page 447 in Kailath~\cite{Kailath} that
$\det P_1(z) \neq 0$ for all $z$ on the unit circle, which
establishes (i'). As an immediated consequence, we have:

\begin{remark} \label{rem-BP-1}
With  the notation of the Theorem~\ref{thm-5} and
\eqref{def-tilde-Q}, assume additionally that $P(z)$ and
$\widetilde{Q}(z)$ are left-coprime. Then condition (i) of
Theorem~\ref{thm-5} is equivalent to $\det P(z) \neq 0$ for all $z$
on the unit circle.
\end{remark}

Next we show how a slight extension of Theorem 4.1 of Bougerol and
Picard~\cite{BP}, which characterized the existence of a strictly stationary
non-anticipative solution of the ARMA$(p,q)$ equation~\eqref{eqpq},
can be deduced from Theorem~\ref{thm-5}. By a {\it
non-anticipative} strictly stationary solution we mean a strictly
stationary solution $Y=(Y_t)_{t\in \bZ}$ such that for every $t\in
\bZ$, $Y_t$ is independent of the sigma algebra generated by
$(Z_s)_{s> t}$, and by a {\it causal} strictly stationary solution
we mean a strictly stationary solution $Y=(Y_t)_{t\in \bZ}$ such
that for every $t\in \bZ$, $Y_t$ is measurable with respect to the
sigma algebra generated by $(Z_s)_{s\leq t}$. Clearly, since
$(Z_t)_{t\in \bZ}$ is assumed to be i.i.d., every causal solution is
also non-anticipative. The equivalence of (i) and (iii) in the
theorem below was already obtained by Bougerol and Picarcd~\cite{BP}
under the additional assumption that $\bE \log^+ \|Z_0\| < \infty$.

\begin{theorem} \label{cor-BP}
In addition to the assumptions and notations of Theorem~\ref{thm-5},
assume that the matrix polynomials $P(z)$ and $\widetilde{Q}(z)$ are
left-coprime, with $\widetilde{Q}(z)$ as defined in
\eqref{def-tilde-Q}. Then the following are equivalent:
\begin{enumerate}
\addtolength{\itemsep}{-1.5ex}
\item[(i)] There exists a non-anticipative strictly stationary
solution of \eqref{eqpq}.
\item[(ii)] There exists a causal strictly stationary
solution of \eqref{eqpq}.
\item[(iii)] $\det P(z) \neq 0$ for all $z\in \bC$ such that
$|z|\leq 1$ and if $M(z) = \sum_{j=0}^\infty M_j z^j$ denotes the
Taylor expansion of $M(z) = P^{-1}(z) \widetilde{Q}(z)$, then
\begin{equation} \label{eq-logfinite2}
\bE \log^+ \| M_j UZ_0\| < \infty \quad \forall\; j \in \{ mp+q-p+1,
\ldots, mp+q\} .
\end{equation}
\end{enumerate}
\end{theorem}

\begin{proof}
 The implication ``(iii)
$\Rightarrow$ (ii)'' is immediate from Theorem~\ref{thm-5} and
equation~\eqref{eq-Y}, and ``(ii) $\Rightarrow$ (i)'' is obvious
since $(Z_t)_{t\in \bZ}$ is i.i.d. Let us show that ``(i)
$\Rightarrow$ (iii)'': since a strictly stationary solution exists,
the function $M(z)$ has only removable singularities on the unit
circle by Theorem~\ref{thm-5}. Since $P(z)$ and $\widetilde{Q}(z)$
are left-coprime, this implies by Remark~\ref{rem-BP-1} that $\det
P(z) \neq 0$ for all $z\in \bC$ such that $|z|=1$. In particular, by
Theorem~\ref{thm-5}, the strictly stationary solution is unique and
given by \eqref{eq-Y}. By assumption, this solution must then be
non-anticipative, so that we conclude that the distribution of $M_j
U Z_{t-j}$ must be degenerate to a constant for all $j\in
\{-1,-2,\ldots\}$. But since $U Z_0 = (w_0^T, u^T)^T$ and $M_j =
(M_j', 0_{m,d-s})$ with certain matrices $M_j' \in \bC^{m,s}$, it
follows for $j\leq -1$ that $M_j UZ_0 = M_j' w_0$, so that $M_j' =
0$ since no non-trivial linear combination of the components of
$w_0$ is constant a.s. It follows that $M_j = 0$ for $j\leq -1$,
i.e. $M(z)$ has only removable singularities for $|z|\leq 1$. Since
$P(z)$ and $\widetilde{Q}(z)$ are assumed to be left-coprime, it
follows from page 447 in Kailath~\cite{Kailath} that $\det P(z) \neq
0$ for all $|z|\leq 1$. Equation \eqref{eq-logfinite2} is an
immediate consequence of Theorem~\ref{thm-5}.
\end{proof}

It may be possible to extend Theorem~\ref{cor-BP} to situations
without assuming that $P(z)$ and $\widetilde{Q}(z)$ are left-coprime,
but we did not investigate this question.

The last result is on the interplay of the existence of strictly and
of weakly stationary solutions of \eqref{eqpq} when the noise is
i.i.d. with finite second moments:

\begin{theorem} \label{cor-strict-weak}
Let $m,d, p\in \bN$, $q\in \bN_0$, and let $(Z_t)_{t\in \bZ}$ be an
i.i.d. sequence of $\bC^d$-valued random vectors with finite second
moment. Let $\Psi_1, \ldots, \Psi_p \in \bC^{m\times m}$ and
$\Theta_0, \ldots, \Theta_q \in \bC^{m\times d}$.  Then the
ARMA$(p,q)$ equation \eqref{eqpq} admits a strictly stationary
solution if and only if it admits a weakly stationary solution, and
in that case, the solution given by \eqref{eq-weakly2} is both a
strictly stationary and weakly stationary solution of \eqref{eqpq}.
\end{theorem}

\begin{proof} It follows from Theorem~\ref{thm-4} that if a weakly
stationary solution exists, then one choice of such a solution is
given by \eqref{eq-weakly2}, which is clearly also strictly
stationary. On the other hand, if a strictly stationary solution
exists, then by Theorem~\ref{thm-5}, one such solution is given by
\eqref{eq-Y}, which is clearly weakly stationary.
\end{proof}

Finally, we remark that most of the results presented in this paper
can be applied also to the case when $(Z_t)_{t\in \bZ}$ is an i.i.d.
sequence of  $\bC^{d\times d'}$ random matrices and $(Y_t)_{t\in
\bZ}$ is $\bC^{m\times d'}$-valued. This can be seen by stacking the
columns of $Z_t$ into a $\bC^{dd'}$-variate random vector $Z_t'$,
those of $Y_t$ into a $\bC^{md'}$-variate random vector $Y_t'$, and
considering the matrices
$$\Psi_k' := \begin{pmatrix} \Psi_k & & \\ & \ddots & \\ & & \Psi_k
\end{pmatrix} \in \bC^{md'\times md'}
\quad \mbox{and} \quad \Theta_k' := \begin{pmatrix} \Theta_k & & \\
& \ddots & \\ & & \Theta_k
\end{pmatrix} \in \bC^{md'\times dd'}.$$
The question of existence of a strictly stationary solution of
\eqref{eqpq} with matrix-valued $Z_t$ and $Y_t$ is then equivalent
to the existence of a strictly stationary solution of $Y_t' -
\sum_{k=1}^p \Psi_k' Y_{t-k}' = \sum_{k=0}^q \Theta_k' Z_{t-k}'$.

\subsubsection*{Acknowledgements}
We would like to thank Jens-Peter Krei{\ss} for helpful comments.
Support from an NTH-grant of the state of Lower Saxony and from
National Science Foundation Grant DMS-1107031 is gratefully
acknowledged.

\end{document}